\documentclass[a4paper]{amsart}

\usepackage[utf8]{inputenc}
\usepackage[british]{babel}
\usepackage{amsfonts,amsmath,amssymb,amsthm,mathtools,esint}
\usepackage{eucal}
\usepackage{tikz}
\usetikzlibrary{arrows}
\usepackage{enumitem}
\usepackage{thmtools} 
\usepackage{cleveref}
\usepackage{comment}
\crefname{subsection}{subsection}{subsections}
\numberwithin{equation}{section}
\usepackage{pgfplots}
\usetikzlibrary{pgfplots.groupplots}
\pgfplotsset{compat=1.3}
\usepackage{xcolor}
\usepackage{caption}

\renewcommand{\div}{\operatorname{div}}
\newcommand{\Fcal}{\mathcal{F}}
\newcommand{\Tcal}{\mathcal{T}}
\newcommand{\pw}{\mathrm{pw}}
\newcommand{\osc}{\operatorname{osc}}
\newtheorem{theorem}{Theorem}[section]
\newtheorem{lemma}[theorem]{Lemma}
\newtheorem{corollary}[theorem]{Corollary}

\theoremstyle{definition}
\newtheorem{definition}[theorem]{Definition}
\newtheorem*{conditionC}{Condition C}
\theoremstyle{remark}
\newtheorem{remark}[theorem]{Remark}
\newtheorem{example}[theorem]{Example}

\begin{document}

\title[Discrete Kirchhoff elements]
      {An error analysis of discrete Kirchhoff elements}

\author[D. Gallistl]{Dietmar Gallistl}
\author[N. T. Tran]{Ngoc Tien Tran}

\thanks{This work received funding from
        the European Union's Horizon 2020 research and innovation 
         programme (project DAFNE, grant agreement No.~891734, 
         and project RandomMultiScales, grant agreement No.~865751).
	}

\address[D. Gallistl]{Institut für Mathematik,
                      Universität Jena, 07743 Jena, Germany}
\email{dietmar.gallistl [at] uni-jena.de}
\address[N. T. Tran]{Institut für Mathematik,
                      Universität Augsburg, 86159 Augsburg, Germany}
\email{ngoc1.tran [at] uni-a.de}

\date{}



\begin{abstract}
The Discrete Kirchhoff Triangle (DKT) method for the biharmonic
equation is analyzed in the discrete energy norm.
The error is bounded by the best approximation of the Hessian 
by piecewise constants and the oscillation of the right-hand side,
without additional regularity assumptions on the exact solution.
This result implies first-order convergence of the classical DKT
element and 
the analysis yields a canonical extension to three space dimensions
with the same approximation properties.
Residual-based a posteriori error estimates are derived.
The analysis is formulated within a general framework for 
low-order nonconforming methods, which also applies to 
various classical elements and yields best-approximation 
results by constants.
It is furthermore shown how known stable pairs for
the planar Stokes system have discrete stream functions in
discrete Kirchhoff spaces. This yields variants of the known
schemes with positive
definite formulations and pressure-robust error bounds.
\end{abstract}

\maketitle

\section{Introduction and main results}

The Discrete Kirchhoff Triangle (DKT) is a finite-element based
numerical scheme for discretizing variational problems posed in subspaces
of the Sobolev space $H^2$, originally invented and developed for
mechanical models of plate bending. Therein, the complicated design
and implementation of conforming finite element methods (FEM)
is circumvented by introducing a separate variable $\theta_h$
representing a discrete gradient. The solution $u$ is approximated
by $u_h$ and its gradient $\nabla u$ by $\theta_h$, and the
quantities $\nabla u_h$ and $\theta_h$ are coupled through
discrete conditions and not through pointwise identity.
The resulting scheme is easy to implement with standard finite
element assembly, once the discrete coupling is encoded
in an additional matrix. In plate analysis, the assumption 
that the displacement gradient $\nabla u$ equals the in-plate
rotation $\theta$ is known as Kirchhoff's hypothesis,
explaining the nomenclature of the DKT element, which enforces
this constraint discretely.
This note is devoted to a structural analysis of the DKT
paradigm for a larger class of low-order discretizations, which
we expose for the biharmonic equation under clamped boundary
conditions for the sake of simplicity of the presentation.

Let $\Omega\subset \mathbb R^{n}$ be an open,
bounded, connected Lipschitz polytope in dimension
$n\in\{2,3\}$, 
with outer unit normal
$\nu$. Given a right-hand side $f\in L^2(\Omega)$,
the weak form of the problem $\Delta^2 u = f$ in $\Omega$ subject
to the clamped boundary condition
$u=\nabla u\cdot\nu=0$ on $\partial \Omega$ seeks 
$u \in V \coloneqq H^2_0(\Omega)$ with 
\begin{align}\label{def:continuous-problem}
    a(u, v) \coloneqq (D^2 u, D^2 v)_{L^2(\Omega)}
    = (f, v)_{L^2(\Omega)} \quad\text{for any } v \in V.
\end{align}
Let $\Tcal$ be a regular triangulation of $\Omega$ into simplices
and let the space $V_h$ consist of piecewise polynomial functions
over $\Omega$, while $\Theta_h$ is a space of piecewise polynomial
vector fields.
Assume we are given a linear discrete gradient map 
$\nabla_h : V_h \to \Theta_h$
that defines the discrete Hessian $D^2_h \coloneqq D_\pw \circ \nabla_h$.
Here and throughout this work, the index $\pw$ attached to a differential
operator denotes its piecewise action with respect to a given mesh
$\Tcal$.
Then the discrete problem seeks a solution $u_h \in V_h$ to
\begin{align}\label{def:discrete-problem}
	a_h(u_h,v_h) = (f,v_h)_{L^2(\Omega)} \quad\text{for any } v_h \in V_h
\end{align}
with the bilinear form
\begin{align*}
	a_h(u_h,v_h) \coloneqq (D_h^2 u_h, D_h^2 v_h)_{L^2(\Omega)}.
\end{align*}
The seminorm induced by $a_h$ is denoted by $\|\cdot\|_{h}$.
If $V_h\subset H^1_0(\Omega)$ is satisfied, we can relax
the conditions on $f$ to $f\in H^{-1}(\Omega)$ in 
\eqref{def:continuous-problem}--\eqref{def:discrete-problem}.

The main result of this note is that the error $\sigma-\sigma_h$
in the $L^2$ norm between $\sigma \coloneqq D^2 u$ and 
$\sigma_h \coloneqq D_h^2 u_h$ is bounded from above by the
best-approximation of $\sigma$ by piecewise constants,
written $\Pi_0\sigma$, plus
oscillations of the right-hand side $f$, provided the following
conditions are satisfied.
In what follows, we denote by $\mathcal F$ the set of faces
of $\Tcal$ and by the brackets $[\cdot]_F$ the jump of a
piecewise polynomial function across a face $F\in\mathcal F$.
For boundary faces, the brackets denote the trace.
The $L^2$ norm over a measurable set $\omega$
is written as $\|\cdot\|_\omega$ with the convention
$\|\cdot\| \coloneqq \|\cdot\|_\Omega$. 

\begin{conditionC}
We assume that any $v_h \in V_h$, any simplex $T \in \Tcal$
with diameter $h_T$, 
and any face $F \in \Fcal$ with diameter $h_F$
satisfy subsequent conditions (C1)--(C5).
\begin{enumerate}[wide,itemsep=0.8em]
\item[(C1)] $\displaystyle \int_F [\nabla_h v_h]_F = 0$

\item[(C2)] $\displaystyle \|[v_h]\|_F 
           \lesssim h_F \|\nabla_t [v_h]_F\|_F$
          (with the tangential gradient $\nabla_t$)

\item[(C3)] $\|D^2_h v_h - D^2 v_h\|_T
\lesssim
\min \left\{
	\|(1 - \Pi_0) D^2 v_h\|_T ,
	\|(1 - \Pi_0) D_h^2 v_h\|_T
\right\}$

\item[(C4)] $\|h_T^{-1}(\nabla v_h - \nabla_h v_h)\|_T
\lesssim \|D^2 v_h - D_h^2 v_h\|_T$

\item[(C5)] 
There exists a bounded linear map $Q_h : V \to V_h$
such that
\begin{align*}
	\sum_{j=0}^2 \|h_T^{j-2}D^{j}(v - Q_h v)\|_T
	\lesssim
	\|(1 - \Pi_0) D^2 v\|_{\omega_T}
  \quad\text{for all }v \in V,\;T \in \Tcal,
\end{align*}
where $\omega_T$, the interior of the set
$\cup\{K\in\Tcal:K\cap T\neq\emptyset\}$,
 is the element patch of $T$.

\end{enumerate}
\end{conditionC}

As a consequence of (C1)--(C4), the bilinear form
$a_h$ is positive definite over $V_h$. In fact, if $D^2_h v_h = 0$,
then $\nabla_h v_h$ is piecewise constant,
(C1) shows that it equals zero.
From (C3), we infer $D^2_\pw v_h \equiv 0$.
Hence, (C4) implies $\nabla_\pw v_h \equiv 0$ and so, 
$v_h$ is piecewise constant. By (C2), $v_h \equiv 0$.
The positivity of $a_h$ shows that \eqref{def:discrete-problem}
possesses a unique discrete solution $u_h$.
The error estimate reads:
\begin{theorem}[a~priori]\label{thm:a-priori}
	Suppose that (C1)--(C5) hold, then the discrete solution
        $u_h \in V_h$ to \eqref{def:discrete-problem} satisfies
	\begin{align*}
          \|\sigma  - D^2_\pw u_h\| + \|\sigma-\sigma_h\| \lesssim \|(1 - \Pi_0)\sigma\|
          + \osc(f,\Tcal).
	\end{align*}
\end{theorem}
The oscillations $\osc(f,\Tcal)$ are described in Section~\ref{s:structure},
\Cref{d:osc}
below and are defined for $f\in H^{-1}$ provided $V_h$ is a subspace
of $H^1_0(\Omega)$. Otherwise, they correspond to the usual $L^2$-based
oscillations.

An immediate consequence is an error bound under minimal regularity
assumptions for several known schemes satisfying Condition~C.
First, the error bound applies to the
DKT element and reveals that the $H^3$ regularity assumed in
prior contributions \cite{Braess2007,Bartels2013,Bartels2015}
is not a qualitative limitation to the method. 
Second, Condition~C as a set of structure conditions
allows for a canonical
generalization of the DKT element to three dimensions,
which leads to a simple low-order scheme
not documented so far in the literature.
The application of the error bounds to several other nonconforming
methods is possible and briefly commented on in Section~\ref{s:conclusion}.
Error estimates in weaker norms are provided in \Cref{t:Hs_estimate}
in Section~\ref{sec:pr-a-priori} below,
and an adaptation of the error bound for a singularly perturbed
problem is given as \Cref{thm:a-priori-singular}.

The conditions furthermore allow for a reliable a~posteriori
error bound.

\begin{theorem}[a~posteriori]\label{thm:a-post}
    Suppose (C1)--(C5) and $f\in L^2(\Omega)$, then the discrete solution 
    $u_h$ to \eqref{def:discrete-problem} satisfies
    \begin{align*}
    \|\sigma - \sigma_h\| \lesssim \mu
    \lesssim
    \|\sigma - \sigma_h\|
    + \|(1-\Pi_0)\sigma\|
    + \| h^2(f-\Pi_0 f)\|
    \end{align*}
    for the explicit residual-based a~posteriori error estimator
    $\mu$ defined in \eqref{e:mudef}.
\end{theorem}

We note that the first bound (reliability) is implied by
Condition~C, while the second bound (efficiency) has been
known \cite{Verfuerth2013} and is not a new contribution here.

Error estimates under minimal regularity assumptions were studied
in \cite{Braess2009,Gudi2010,CarstensenPeterseimSchedensack2012}
and more abstractly in \cite{VeeserZanotti2018_I}
for classical nonconforming or discontinuous Galerkin schemes.
The error analysis in this work shows that DKT elements and many
other known nonconforming methods for the biharmonic operator share
structural conditions that are sufficient for quasi-best approximation
of the Hessian by piecewise constants.
We furthermore describe in \Cref{s:stream}
a connection to two-dimensional stable pairs
for the Stokes equations. This gives room for a re-interpretation
of DKT-like elements as discrete stream functions of known discretely
divergence-free functions, rather than an ad-hoc construction for
plate analysis only.

\medskip
Throughout this article, standard notation on Lebesque
and Sobolev spaces is employed
with the $L^2$ inner product $(\cdot,\cdot)_\omega$
over a measurable subset $\omega$ of $\mathbb R^n$
and the $L^2$ norm $\|\cdot\|_\omega$.
If $\omega=\Omega$, the index is dropped.
The space of functions over a set $T$ that are polynomial
of degree at most $k$ is denoted as $P_k(T)$,
and $P_k(\Tcal)$ denotes the space of functions that 
belong to $P_k(T)$ when restricted to any simplex $T$
of the triangulation $\Tcal$.
The $L^2$ projection onto $P_0(\mathcal T)$ is denoted
by $\Pi_0$.
The diameter of a set $\omega$ is denoted by $h_\omega$,
and $h$ denotes the piecewise constant mesh-size function
with respect to a triangulation $\Tcal$ such that
$h|_T=h_T$ for any simplex $T\in\Tcal$; $h_{\max} = \|h\|_{L^\infty(\Omega)}$ 
is the maximal mesh-size.
An inequality $A\leq CB$ with a generic constant $C$ that
may depend on the shapes in the triangulation $\Tcal$
but not on the mesh size, is denoted by $A\lesssim B$,
and we write $A\approx B$ for $A\lesssim B \lesssim A$.
The Frobenius inner product of matrices $M,N$ is denoted
by the colon, $M:N$.

The remaining parts of this paper are organized as follows. 
\Cref{s:structure} discusses several preliminary results implied
by Condition~C, including the construction of smoothing operators.
\Cref{sec:pr-a-priori} is devoted to the proof of \Cref{thm:a-priori}
as well as its consequences such as lower-order estimates and 
extension to singularly perturbed biharmonic operators. The proof 
of \Cref{thm:a-post} is given in \Cref{sec:pr-a-post}. 
The  results are applied to DKT elements in \Cref{s:dkt}. 
\Cref{s:stream} discusses the relation to pressure-robust
discretizations of planar Stokes flow.
Two and three dimensional numerical benchmarks are provided in \Cref{s:num}. Some comments on other classical nonconforming FEM in \Cref{s:conclusion} conclude this note.

\section{Consequences of the structure conditions}
\label{s:structure}

In this section, we derive some direct conclusions
from Condition~C
and prove the existence of a smoothing operator $J$.
This operator allows to quantify consistency and oscillations
of the right-hand side.
Such operators were constructed in more specific situations
by \cite{BrennerSung2005,Gallistl2015,VeeserZanotti2019,Tran2026}, for example.
As a preliminary step, we discuss some basic properties around
(C1)--(C5) on local equivalence of the classical gradient
$\nabla$ and the discrete gradient $\nabla_h$ acting on
discrete functions from $V_h$.
A first consequence of (C3) is the equivalence of the local 
seminorms 
\begin{equation}\label{e:normeq}
\|D^2 \bullet\|_{T} \approx \|D^2_h \bullet\|_{T}
\end{equation}
for discrete functions.
Condition (C3) is satisfied in most reasonable discrete schemes
because of the following.
Suppose that
$\nabla_h v_h|_T$ being affine implies
that $v_h|_T$ is a quadratic polynomial
and $\nabla_h v_h|_T = \nabla v_h|_T$.
Then, (C3) is satisfied.
The hidden constant therein depends on the shape of the simplex.
Under the condition that these constants are invariant under 
scaling and translation, they remain uniformly bounded for families of triangulations involving
finite many shapes, such as the ones created by the newest-vertex bisection.
Another consequence of (C3)--(C4) is the local equivalence of the gradient
and the discrete gradient, namely
\begin{equation}\label{e:normeq_nabla}
\|\nabla \bullet\|_{T} \approx \|\nabla_h \bullet\|_{T},
\end{equation}
which can be proven by combining (C4) with the triangle inequality,
(C3), and an inverse estimate.
We further note that (C3) implies the following local equivalence
\begin{equation}\label{e:ba-loc}
	\|(1 - \Pi_0) D^2 v_h\|_T \approx \|(1 - \Pi_0) D_h^2 v_h\|_T.
\end{equation}
 
 \begin{lemma}[best approximation of $D^2_h \circ Q_h$]\label{lem:ba-dis-Hessian}
 	Suppose (C3) and (C5), then any $v \in V$ satisfies
 	\begin{align*}
 		\|D^2 v - D_h^2 Q_h v\|_T \lesssim \|(1 - \Pi_0) D^2 v\|_{\omega_T}.
 	\end{align*}
 \end{lemma}
 \begin{proof}
 	From the triangle inequality, (C5), and (C3), we deduce that
 	\begin{align*}
 		\|D^2 v - D_h^2 Q_h v\|_T 
 		&\leq \|D^2(v - Q_h v)\|_T + \|D^2 Q_h v - D^2_h Q_h v\|_T\\
 		&\lesssim \|(1-\Pi_0) D^2 v\|_{\omega_T} 
 		+ \|(1 - \Pi_0) D^2 Q_h v\|_T 
 		\lesssim \|(1 - \Pi_0) D^2 v\|_{\omega_T}.
 	\end{align*}
 \end{proof}
Let $\{\cdot\}_F$ denote the average of the values from the
simplices adjacent to $F \in \Fcal$ (for boundary faces it denotes the trace)
and $\nu_F$ is a given normal unit vector of $F$ (for boundary faces it coincides with the outer normal unit vector $\nu$ of $\Omega$).
\begin{lemma}[existence of smoothing]
\label{l:smoothing}
Suppose (C1)--(C4), then there exists a linear bounded operator
$J : V_h \to V$ such that any $v_h \in V_h$ satisfies
\begin{align*}
	\sum_{j=0}^2 \|h_T^{j-2}D^j(v_h - Jv_h)\|_T
	\lesssim 
         \min_{\psi \in V} \|D_\pw^2 (\psi - v_h)\|_{\omega_T} +
	\|(1 - \Pi_0) D^2_\pw v_h\|_{\omega_T},
\end{align*}
as well as
\begin{align}\label{eq:J-orth}
 \int_T (J v_h - v_h) p = 0 
   \quad\text{and}\quad
 \int_F (\nabla J v_h - \{\nabla_h v_h\}_F) \cdot \nu_F = 0
\end{align}
for all $T\in\Tcal$, all $F\in \Fcal$, and any piecewise
quadratic function $p$ with respect to $\Tcal$.
\end{lemma}
\begin{proof}
The design of $J$ departs from standard averaging techniques of 
the degrees of freedom of (higher-order) $C^1$ finite element functions on 
Hsieh--Clough--Tocher splits $\widehat{\Tcal}$ if $n=2$
or Worsey--Farin splits $\widehat{\Tcal}$ if $n=3$.
Since the degrees of freedom only depend on the evaluation of a function and
its first derivatives \cite{HCT,GuzmanLischkeNeilan2022}, there exists, for
any given $v_h \in V_h$, a conforming finite element function 
$v_c \in V$, piecewise polynomial with respect to $\hat{\Tcal}$ of 
sufficiently high polynomial degree,
satisfying the local approximation property
\begin{align*}
	h_T^{-4}\|v_h - v_c\|_T^2 \lesssim \sum_{F \in \mathcal{F}, F \cap T \neq \emptyset} (h_F^{-3}\|[v_h]\|_F^2 + h_F^{-1}\|[\nabla_\pw v_h \cdot \nu_F]_F\|_F^2)
\end{align*}
for any $T \in \Tcal$. This is bounded by gradient jumps due to (C2),
\begin{align*}
	h_T^{-4}\|v_h - v_c\|_T^2
	\lesssim \sum_{F \in \mathcal{F}, F \cap T \neq \emptyset}
	h_F^{-1}\|[\nabla_\pw v_h]_F\|_F^2.
\end{align*}
The triangle and trace inequality, (C4), and (C3) imply
\begin{align}\label{ineq:v_h-v_c}
\begin{aligned}
	h_T^{-4}\|v_h - v_c\|_T^2 &\lesssim \sum_{F \in \mathcal{F}, F \cap T \neq \emptyset}
	h_F^{-1}\|[\nabla_h v_h]_F\|_F^2 + \|(1-\Pi_0)D^2_\pw v_h\|^2_{\omega_T}\\
	&\lesssim \min_{\psi \in V} \|D_\pw^2 (\psi - v_h)\|_{\omega_T}^2 + \|(1-\Pi_0)D^2_\pw v_h\|^2_{\omega_T},
\end{aligned}
\end{align}
where we utilize
(C1), standard bubble function techniques
\cite{Verfuerth2013,HuShi2009}, and \eqref{e:ba-loc} in the last step.
The degrees of freedom of conforming finite elements are known from 
\cite{Zenisek1973,Zhang2009} and can be written in integral form 
\cite{LiangTran2026}.
By prescribing these degrees of freedom
as in \cite[Section 2.3]{Tran2026}, we find a piecewise polynomial function
$w_c \in V$ of degree $8$ if $n=2$ or $10$ if 
$n=3$ with respect to $\Tcal$
such that
\begin{align*}
	\int_F \nabla  J w_c\cdot\nu_F
    &= \int_F \{\nabla_h v_h - \nabla v_c\}_F \cdot \nu_F
  \quad\text{and}\quad
	\int_T w_c q  = \int_T (v_h - v_c) q
\end{align*}
for any $F \in \Fcal$,
any $T\in \Tcal$, and any piecewise quadratic function $q \in P_2(\Tcal)$, and
\begin{align*}
	h_{T}^{-4}\|w_c\|^2_{T}
       \lesssim h_{T}^{-4}\|v_h - v_c\|^2_{T} 
      + \sum_{F \in \Fcal,F\subset \partial T} h_F^{-1}\|\{\nabla_h v_h - \nabla v_c\} 
                                                    \cdot \nu_F\|_{F}^2. 
\end{align*}
The triangle and trace inequalities as well as
inverse estimates thus imply
\begin{align}\label{ineq:corrector-bound}
	h_{T}^{-4}\|w_c\|^2_T \lesssim h_T^{-4}\|v_h - v_c\|^2_T + h_T^{-2}\|\nabla_h v_h - \nabla v_c\|^2_{\omega_T}.
\end{align}
The function $J v_h \coloneqq v_c + w_c \in V$ satisfies \eqref{eq:J-orth}.
From \eqref{ineq:v_h-v_c}--\eqref{ineq:corrector-bound}, 
the equivalence \eqref{e:ba-loc}, and inverse estimates, we infer the asserted error bound.
The continuity of $J$ follows from this and (C3) and (C4), namely,
\begin{align}\label{ineq:appr-smoothing}
 \begin{aligned}
	h_T^{-1}\|\nabla_h v_h &- \nabla J v_h\|_T + \|D^2_h v_h - D^2 J v_h\|_T\\
	&\lesssim \min_{\psi \in V} \|D_\pw^2 (\psi - v_h)\|_{\omega_T} +
	\|(1 - \Pi_0) D^2_\pw v_h\|_{\omega_T}.
 \end{aligned}
\end{align}
\end{proof}
As a direct consequence of the foregoing lemma and (C3), we note that
the map $J \circ Q_h$ 
defines a quasi-interpolation into $C^1$ conforming finite 
elements with local approximation properties.

\begin{lemma}[$C^1$ quasi-interpolation] 
\label{l:C1qi} 
Under conditions (C1)--(C5),
\begin{align*}
	\sum_{j = 0}^2 \|h^{j-2}D^j(v - JQ_h v)\| \lesssim \|(1 - \Pi_0) D^2 v\| \quad\text{for any } v \in V.
\end{align*}
\end{lemma}
\begin{proof}
The triangle inequality
\begin{align*}
	\|D^j(v - JQ_h v)\|_T 
	\leq \|D^j(v - Q_h v)\|_T + \|D^j(Q_h v - J Q_h v)\|_T
\end{align*}
followed by an application of \Cref{l:smoothing} and (C5)
shows
\begin{align*}
	\sum_{j = 0}^2 h_T^{j-2}\|D^j(v - JQ_h v)\|_T 
    \lesssim \|(1 - \Pi_0) D^2 v\|_{\omega(\omega_T)}
      + \|(1 - \Pi_0) D^2_\pw Q_h v\|_{\omega_T}.
\end{align*}
Here, $\omega(\omega_T)$ is the second-order patch of $T$,
that is, the domain consisting of the elements
inside or surrounding the closure of the patch $\omega_T$.
The final term is controlled by 
$\|(1 - \Pi_0) D^2 v\|_{\omega(\omega_T)}$ from (C5) and a triangle inequality, and the claim ensues.
\end{proof}

The following result is the main ingredient for the error analysis.
\begin{lemma}[consistency with lowest-order test functions]\label{lem:dis-cons}
Suppose (C1)--(C4), then any $v_h,w_h \in V_h$ and any piecewise quadratic
function $p$ satisfy
\begin{align*}
	&\big|(D_\pw^2 p, D_h^2 v_h - D^2 J v_h)_{\Omega}\big|\\
	&\qquad\lesssim 
	\|D^2_\pw (p - J w_h)\| \Big(\min_{\psi \in V} \|D_\pw^2 (\psi - v_h)\| +
	\|(1 - \Pi_0) D^2_\pw v_h\|\Big).
\end{align*}
\end{lemma}
\begin{proof}
	Let $v_h \in V_h$ and a piecewise quadratic $p$
        be given.
	Since the jumps of discrete gradients $\nabla_h v_h$
	have zero averages over the faces due to (C1),
	a piecewise integration by parts shows
	\begin{align}\label{eq:p2-cons-proof}
		(D_\pw^2 p, D_h^2 v_h - D^2 J v_h)_\Omega
	 	= \sum_{F \in \Fcal} \int_F \{\nabla_h v_h - \nabla J v_h\}_F \cdot [D_\pw^2 p \nu_F]_F
        .
	\end{align}
Given any face $F \in \Fcal$, the orthogonality
relation \eqref{eq:J-orth} proves that the integral of
$
  (\nabla_h v_h - \nabla J v_h) \cdot \nu_F
$
over $F$ against any constant vanishes.
Choosing the constant $ \nu_F \cdot [D_\pw^2 p \nu_F]_F$,
we infer
\begin{align*}
\int_F \{\nabla_h v_h - \nabla J v_h\}_F \cdot [D_\pw^2 p \nu_F]_F 
= \int_F \sum_{j = 1}^{n-1} (
   \{\nabla_h v_h - \nabla J v_h\}_F \cdot t_j) [\partial_{nt_j} p]_F 
\end{align*}
with $n-1$ orthonormal vectors $t_1, \dots, t_{n-1}$ spanning the 
hyperplane $F$. Along tangential directions, gradients of the $C^1$ conforming 
function $J w_h$ do not jump and are zero on boundary faces. Hence,
\begin{align*}
  \int_F \{\nabla_h v_h - \nabla J v_h\}_F \cdot [D_\pw^2 p \nu_F]_F
  \lesssim 
  \|\{\nabla_h v_h - \nabla J v_h\}_F\|_{F}
   \|[\partial_{nt} (p - J w_h)]_F\|_{F} 
   .
\end{align*}
The sum of this over all $F \in \Fcal$ and the discrete trace inequality lead to
\begin{align}\label{ineq:p2-cons-pr}
 \begin{aligned}
   &\sum_{F \in \Fcal} \int_F \{\nabla_h v_h - \nabla J v_h\}_F
                \cdot [D_\pw^2 p \nu_F]_F \\
    &\qquad \lesssim 
       \left(
           \|h^{-1}(\nabla_h v_h - \nabla J v_h)\| 
          + \|D_h^2 v_h - D^2 J v_h\|
		\right)
		\|D^2_\pw(p - J w_h)\|.
 \end{aligned}
\end{align}
The combination of \eqref{ineq:appr-smoothing}--\eqref{ineq:p2-cons-pr}
with the finite overlap of patches implied by the shape regularity
concludes the proof.
\end{proof}
\begin{remark}[gradient consistency]\label{rem:mod-smoothing}
	If we additionally assume 
	continuity $V_h \subset H^1_0(\Omega)$,
	which implies (C2),
	then the operator $J$ from \Cref{l:smoothing} can be chosen to
	satisfy
	\begin{align}\label{e:gradcons_singpert}
		(\nabla_\pw (v_h - J v_h), \Phi)_\Omega = 0.
	\end{align}
	for any $v_h \in V_h$ and any piecewise constant vector field $\Phi$.
	To this end, we further enforce the property	
	\begin{align}\label{eq:J-orth-singular}
		\int_F (J v_h - \{v_h\}_F) = 0
		\quad\text{for all } F\in \Fcal
	\end{align}
	in \Cref{l:smoothing} by face bubble functions on interior faces.
	On the boundary, the operations $\{\cdot\}_F$ and $[\cdot]_F$ coincide and
	\eqref{eq:J-orth-singular} is directly implied by the
	assumed inclusion in $H^1_0(\Omega)$.
	A piecewise integration by parts proves that 
	$(\nabla_\pw (v_h - J v_h), \Phi)_\Omega$ is equal to
	\begin{align*}
		\sum_{F \in \Fcal} \int_F [v_h]_F \{\Phi\}_F \cdot \nu_F 
		+ \sum_{F \in \Fcal} \int_F (\{v_h\}_F - J v_h) [\Phi]_F \cdot \nu_F.
	\end{align*}
	This vanishes by the assumed $H^1_0$ property
	and \eqref{eq:J-orth-singular}, implying \eqref{e:gradcons_singpert}.
	Notice that continuity of lowest-order moments
	is sufficient for \eqref{e:gradcons_singpert}, but pointwise
	continuity will be utilized in the discussion on singular perturbed problems
	in \Cref{thm:a-priori-singular} below.
\end{remark}

Based on the operator $J$, we define the data oscillation
$\osc(f,\Tcal)$ of $f$ as follows.

\begin{definition}[oscillation]
\label{d:osc}
We define
\begin{align*}
\osc(f,\Tcal)\coloneqq
 \sup_{w_h \in V_h, \|w_h\|_{h} = 1} \int_\Omega f (w_h - J w_h) ,
\end{align*}
where in the case of $f\in H^{-1}(\Omega)$ and
$V_h\subset H^1_0(\Omega)$, the integral on the right-hand side
is read as the duality pairing.
\end{definition}

The following result shows that the classical data
oscillation provides an upper bound of the quantity
defined in \Cref{d:osc}.

\begin{lemma}[oscillation]\label{lem:osc}
Under conditions (C1)--(C4), the oscillation satisfies the 
following bounds. If $f\in L^2(\Omega)$, then
\begin{align*}
\osc(f,\Tcal)\lesssim \|h^2(1 - \Pi_0) f\|.
\end{align*}
If $f \in H^{-1}(\Omega)$ and $V_h \subset H^1_0(\Omega)$, then
\begin{align*}
    \osc(f,\Tcal)\lesssim 
          \Big(\sum_{z \in \mathcal{V}} 
                h_{\omega_z}^2\|f\|_{H^{-1}(\omega_z)}^2\Big)^{1/2}
\end{align*}
with $\mathcal V$ being the set of vertices and $\omega_z$ the
vertex patch of a vertex $z \in \mathcal{V}$
with diameter $h_{\omega_z}$.
\end{lemma}
\begin{proof}
	For $f\in L^2(\Omega)$, the upper bound by the $L^2$ oscillation follows
	from the bound of \Cref{l:smoothing}, the 
	orthogonality \eqref{eq:J-orth},
	and the continuity of $J$.
	If $f \in H^{-1}(\Omega)$, 
	the bound  by a localized version of $h_{\max}\|f\|_{H^{-1}(\Omega)}$ 
	is as follows.
	For any vertex $z \in \mathcal{V}$, let
	$\Lambda_z$ denote the piecewise affine and globally continuous
	nodal hat function associated with $z$.
	Then
 \begin{align*}
	\langle f, w_h - J w_h\rangle
        = \sum_{z \in \mathcal{V}} \langle f, \Lambda_z(w_h - J w_h) \rangle
	\leq \sum_{z \in \mathcal{V}} \|f\|_{H^{-1}(\omega_z)} \|\nabla (\Lambda_z(w_h - J w_h))\|_{\omega_z}
\end{align*}
where the angle brackets denote duality between $H^{-1}(\Omega)$ and $H^1_0(\Omega)$.
The product rule implies
$
\nabla( \Lambda_z(w_h - J w_h)) 
=
\Lambda_z \nabla(w_h - J w_h) + (w_h - J w_h)\nabla \Lambda_z
$.
The        
triangle inequality, an inverse estimate,
the scaling $\|\nabla\Lambda_z\|_\infty\lesssim h_{\omega_z}^{-1}$,
and
   \Cref{l:smoothing} provide
	\begin{align*}
		\|\nabla( \Lambda_z(w_h - J w_h))\|_{\omega_z} 
              \lesssim h_{\omega_z}\|D^2_\pw w_h\|_{\omega(\omega_z)}
	\end{align*}
	with the patch $\omega(\omega_z)$ of $\omega_z$. 
	The combination of the two previously displayed formula with 
        the Cauchy inequality and the discrete norm equivalence 
        \eqref{e:normeq}
     prove 
	\begin{align*}
		\langle f, w_h -J w_h\rangle
          \lesssim \Big(\sum_{z \in \mathcal{V}} h_{\omega_z}^2
                    \|f\|_{H^{-1}(\omega_z)}^2\Big)^{1/2} \|w_h\|_{h}
	\end{align*}
and thus the asserted bound.
\end{proof}

\section{A~priori error analysis}\label{sec:pr-a-priori}

This section is devoted to the proof of \Cref{thm:a-priori}
and extensions of the error analysis to weaker norms (\Cref{t:Hs_estimate})
and a singular perturbation problem (\Cref{thm:a-priori-singular}).
Recall the notation $\sigma=D^2 u$ and $\sigma_h=D_\pw\nabla_h u_h=D^2_h u_h$
for the solutions $u$ and $u_h$ to \eqref{def:continuous-problem} and
\eqref{def:discrete-problem}.
The proof of \Cref{thm:a-priori} utilizes the following Galerkin projection onto the space of piecewise quadratic functions.
For any $v \in V$, the piecewise quadratic function
$G_h v \in P_2(\Tcal)$ is the unique solution to 
\begin{align*}
	(D^2_\pw G_h v, D^2_\pw p)_\Omega &= (D^2 v, D^2_\pw p)_{\Omega} &&\text{for all } p \in P_2(\Tcal),\\
	(G_h v, q)_\Omega &= (v, q)_\Omega &&\text{for all } q \in P_1(\mathcal{T}).
\end{align*}
By design, 
$\|D^2(v - G_h v)\| = \min_{p \in P_2(\Tcal)} \|D^2_\pw(v - p)\|$ and,
since the space of piecewise Hessians of the piecewise quadratic functions
equals the space of piecewise constant symmetric-matrix valued
functions, written
$D^2_\pw P_2(\mathcal{T}) = P_0(\mathcal{T};\mathbb S)$, and $D^2 v$ is symmetric pointwise a.e.,
we deduce
\begin{align}\label{eq:Galerkin-err}
	\|D^2(v - G_h v)\| = \|(1 - \Pi_0) D^2 v\|.
\end{align}
We proceed with the proof of \Cref{thm:a-priori}.

\begin{proof}[Proof of \Cref{thm:a-priori}]
	With the abbreviation $e_h = Q_h u - u_h$ and \eqref{def:continuous-problem}--\eqref{def:discrete-problem}, the proof departs from the split
	\begin{align}\label{eq:pr-a-priori-split}
		\|e_h\|^2_h = a_h(Q_h u - u_h, e_h) = a_h(Q_h u, e_h) - a(u, J e_h) + (f, J e_h - e_h)_\Omega. 
	\end{align}
	Elementary algebra proves
	\begin{align*}
		a_h(Q_h u, e_h) - a(u, J e_h) = (D_h^2 Q_h u, D_h^2 e_h - D^2 J e_h)_\Omega - (\sigma - D_h^2 Q_h u, D^2 J e_h)_\Omega.
	\end{align*}
	\Cref{lem:dis-cons} (with $p = G_h u$, $v_h = e_h$, $w_h = Q_h u$, $\psi = 0$)
    and \eqref{e:normeq} provide
    \begin{align*}
		(D_\pw^2 G_h u, D_h^2 e_h - D^2 J e_h)_\Omega 
		\lesssim  \|D^2_\pw (G_h u - J Q_h u)\|\|e_h\|_h
    \end{align*}
     and so, with the 
     continuity of $J$, 
	\begin{align*}
		&(D_h^2 Q_h u, D_h^2 e_h - D^2 J e_h)_\Omega 
		\lesssim \big(\|D_h^2 Q_h u - D^2_\pw G_h u\| 
		+ \|D^2_\pw (G_h u - JQ_h u)\|\big)\|e_h\|_h.
	\end{align*}
	Since 
        $
         \|\sigma - D^2_h Q_h u\| + \|\sigma - D^2_\pw G_h u\| 
        + \|\sigma - D^2 J Q_h u\| 
        \lesssim \|(1 - \Pi_0) \sigma\|
       $
       from \Cref{lem:ba-dis-Hessian}, \eqref{eq:Galerkin-err}, and \Cref{l:C1qi}, 
       the three previously displayed formula and triangle inequality imply
	\begin{align*}
		a_h(Q_h u, e_h) - a(u, J e_h) \lesssim \|(1 - \Pi_0) \sigma\|\|e_h\|_h.
	\end{align*}
	The combination of this with
	\Cref{d:osc} and \eqref{eq:pr-a-priori-split} results in
	\begin{align*}
		\|e_h\|_h \lesssim \|(1 - \Pi_0) \sigma\| + \osc(f,\Tcal).
	\end{align*}
 The error bound for $\sigma-\sigma_h$ follows from the 
 triangle inequality, the equivalence \eqref{e:normeq},
 the bound on  $e_h$, and the bound (C5).
To conclude the proof, we observe that $\|\sigma - D^2_\pw u_h\| \lesssim \|\sigma - \sigma_h\| + \|(1-\Pi_0)\sigma_h\| \leq 2\|\sigma - \sigma_h\| + \|(1 - \Pi_0) \sigma\|$ from (C3) and the triangle inequality.
\end{proof}
A direct consequence of \Cref{thm:a-priori} is the estimate
\begin{align}\label{ineq:alt-err-est}
	\|\sigma_h - D^2 J u_h\| + \|\sigma - D^2 J u_h\| \lesssim \|(1 - \Pi_0) \sigma\| + \osc(f,\Tcal)
\end{align}
with the triangle inequality, \Cref{l:C1qi}, and the continuity of $J$.
This shows that $D^2 J u_h$ is a feasible approximation of $\sigma$ and $\|\sigma_h - D^2 J u_h\|$ is an efficient contribution of a~posteriori error estimators. Another consequence of \Cref{thm:a-priori} is the quasi-optimality, up to data oscillations, of $\|\sigma - \bar\sigma_h\|$ with the piecewise constant function $\bar\sigma_h \coloneqq \Pi_0 \sigma$, namely
\begin{align}\label{ineq:q-opt-pi0}
	\|\sigma - \bar\sigma_h\| \lesssim \|(1 - \Pi_0) \sigma\| + \osc(f,\Tcal).
\end{align}
Other consequences of \Cref{thm:a-priori} are lower-order estimates as well as error estimates for the singular perturbed problem. Regarding the former,
we focus, for the sake of brevity, on $H^1$ estimates and mention that the arguments 
below also enable $H^s$ estimates for any $0 \leq s < 2$.
For low-order methods we do not expect higher convergence rates in norms weaker
than $H^1$, cf.~\cite{HuShi2012}.
We assume the existence of $0 < \delta \leq 1$ such that the solution to
$\Delta^2 z = g$ for any given $g \in H^{-1}(\Omega)$ enjoys the elliptic regularity $z \in H^{2 + \delta}(\Omega)$ with
\begin{align}\label{ineq:elliptic-regularity}
	\|z\|_{H^{2 + \delta}} \lesssim \|g\|_{H^{-1}(\Omega)}.
\end{align} 
In the two-dimensional case, such estimates are provided in
\cite{BlumRannacher1980,Grisvard1985}. 
Recall the data oscillation $\operatorname{osc}(f,\Tcal)$ from \Cref{d:osc}. 
Upper bounds are given in \Cref{lem:osc} and we abbreviate them 
by $\widetilde{\operatorname{osc}}(f,\Tcal)$.
\begin{corollary}[gradient error]
 \label{t:Hs_estimate}
	Suppose (C1)--(C5) and the regularity assumption
        \eqref{ineq:elliptic-regularity}.
        Then the discrete solution $u_h$ to \eqref{def:discrete-problem} satisfies
	\begin{align*}
		\|\nabla_\pw(u - u_h)\| \lesssim h_{\max}^\delta(\|(1 - \Pi_0) \sigma\| 
               + \widetilde{\osc}(f,\Tcal)).
	\end{align*}
\end{corollary}
\begin{proof}
	Let $z \in V \cap H^{2+\delta}(\Omega)$ solve $\Delta^2 z = - \Delta (u - J u_h)$. An integration by parts and the variational formulation of this yield
	\begin{align*}
		\|\nabla(u - J u_h)\|^2 = -(u - J u_h,\Delta(u - J u_h))_\Omega = a(z,u-J u_h).
	\end{align*}
	An elementary algebraic split then shows
	\begin{align}\label{eq:pr-Hs-split}
         	\|\nabla(u - J u_h)\|^2
            =  a(z - J Q_h z, u - J u_h) + a(J Q_h z, u - J u_h).
	\end{align}
	The solution property of $u$ and $u_h$ from \eqref{def:continuous-problem}--\eqref{def:discrete-problem} prove
        for the second term on the right-hand side of \eqref{eq:pr-Hs-split} that
	\begin{align}\label{eq:pr-Hs-T1-split}
         \begin{aligned}
		a(J Q_h z, u - J u_h) = (f, J Q_h z)_\Omega - a(J Q_h z, J u_h) = (f, J Q_h z - Q_h z)_\Omega&\\
		- (D^2 J Q_h z - D_h^2 Q_h z, D^2 J u_h)_\Omega - (D_h^2 Q_h z, D^2 J u_h - \sigma_h)_\Omega&. 
         \end{aligned}
	\end{align}	
	We bound the last two terms on the right-hand side
        of \eqref{eq:pr-Hs-T1-split} as follows.
	Since $\|(1 - \Pi_0) D^2_\pw Q_h z\| \lesssim \|(1 - \Pi_0) D^2 z\|$ from a triangle inequality and (C5),
	\Cref{lem:dis-cons} (with $p = G_h u$, $v_h = Q_h z$, $w_h = u_h$, $\psi = z$) implies
	\begin{align*}
		(D^2 J Q_h z - D_h^2 Q_h z, D^2_\pw G_h u)_\Omega
          \lesssim \|(1 - \Pi_0) D^2 z\|\|D^2_\pw (G_h u - J u_h)\|.
	\end{align*}
	Together with $\|D^2 J Q_h z - D_h^2 Q_h z\| \lesssim \|(1 - \Pi_0) D^2 z\|$ from \Cref{lem:ba-dis-Hessian} and \Cref{l:C1qi} as well as $\|D^2_\pw (G_h u - J u_h)\| \lesssim \|(1 - \Pi_0) \sigma\| + \operatorname{osc}(f,\Tcal)$ from \eqref{eq:Galerkin-err} and \eqref{ineq:alt-err-est},
	\begin{align}\label{ineq:pr-Hs-T11}
		- (D^2 J Q_h z - D_h^2 Q_h z, D^2 J u_h)_\Omega
            \lesssim \|(1 - \Pi_0) D^2 z\|(\|(1 - \Pi_0) \sigma\| + \operatorname{osc}(f,\Tcal)) 
	\end{align}
	ensues.
	We infer from \Cref{lem:dis-cons} (with $p = G_h z$, $v_h = u_h$, $w_h = Q_h z$, $\psi = u$), $\|(1 - \Pi_0) D^2_\pw u_h\| \leq \|\sigma - D^2_\pw u_h\| + \|(1 - \Pi_0) \sigma\|$, and \Cref{thm:a-priori} that
	\begin{align*}
		(D_\pw^2 G_h z, D^2 J u_h - \sigma_h)_\Omega
               \lesssim \|D_\pw^2 (G_h z - J Q_h z)\|(\|(1 - \Pi_0) \sigma\| + \operatorname{osc}(f,\Tcal)).
	\end{align*}
	Since 
       $\|D_\pw^2 (G_h z - J Q_h z)\| + \|D_h^2 Q_h z - D^2_\pw G_h z\| \lesssim \|(1 - \Pi_0) D^2 z\|$ 
       from the triangle inequality, \eqref{eq:Galerkin-err}, 
       \Cref{lem:ba-dis-Hessian}, and \Cref{l:C1qi},
       we see that
	\begin{align}\label{ineq:pr-Hs-T12}
		-(D_h^2 Q_h z, D^2 J u_h - \sigma_h)_\Omega 
          \lesssim \|(1 - \Pi_0) D^2 z\| (\|(1 - \Pi_0) \sigma\| + \operatorname{osc}(f,\Tcal)) 
	\end{align}
	follows for the third term on the right-hand side of
        \eqref{eq:pr-Hs-T1-split}
        from the previously displayed formula and \eqref{ineq:alt-err-est}.
        Since $\|h^{j-2}D^j(J Q_h z - Q_h z)\| \lesssim \|D^2(J Q_h z - Q_h z)\|$ for $j \in \{0,1\}$ from the Poincar\'e inequality with \eqref{eq:J-orth}, 
      	we bound the first term on the right-hand
      	side, following the argumentation of \Cref{lem:osc}, by 
	\begin{align*}
		(f, J Q_h z - Q_h z)_\Omega
		\lesssim \widetilde{\osc}(f,\Tcal) \|D^2(J Q_h z - Q_h z)\| \lesssim \widetilde{\osc}(f,\Tcal) \|(1 - \Pi_0)D^2 z\|,
	\end{align*}
	where \Cref{l:C1qi} and (C5) is utilized in the final step.
	Hence,
	the combination of \eqref{eq:pr-Hs-T1-split}--\eqref{ineq:pr-Hs-T12} results in
	\begin{align}\label{ineq:pr-Hs-T1}
		a(J Q_h z, u - J u_h)
              \lesssim \|(1 - \Pi_0) D^2 z\|
              \left(\|(1 - \Pi_0) \sigma\| + \widetilde{\osc}(f,\Tcal)\right).
	\end{align}
	The right hand-side also controls the first term on the right-hand side of
        \eqref {eq:pr-Hs-split}	
        by a Cauchy inequality, \Cref{l:C1qi}, and \eqref{ineq:alt-err-est}.
	This, \eqref{eq:pr-Hs-split}, \eqref{ineq:pr-Hs-T1}, and the elliptic regularity \eqref{ineq:elliptic-regularity} of $z$ imply
	\begin{align*}
		\|\nabla(u - J u_h)\| \lesssim h_{\max}^{\delta}(\|(1 - \Pi_0) \sigma\| 
                 + \widetilde{\osc}(f,\Tcal)).
	\end{align*}
	From \Cref{l:smoothing}, we deduce that
	\begin{align*}
		\|\nabla_\pw(u_h - J u_h)\| \lesssim h_{\max}(\|\sigma - D^2_\pw u_h\| + \|(1 - \Pi_0) \sigma\|).
	\end{align*}
	The two previously displayed formula, \Cref{thm:a-priori}, 
and the triangle inequality conclude the assertion.
\end{proof}

In the remaining parts of this section, we extend the a~priori
error bound of \Cref{thm:a-priori} to the case of a parameter-dependent problem.
Given $\varepsilon > 0$, the singular perturbed biharmonic 
problem seeks the unique solution $u \in V$ to
\begin{align}\label{def:a-singular-perturbed}
	a_\varepsilon(u,v) 
        \coloneqq \varepsilon^2 (D^2 u, D^2 v)_\Omega + (\nabla u, \nabla v)_\Omega = 
         (f,v)_\Omega \quad\text{for any } v \in V.
\end{align}
The scalar product $a_\varepsilon$ induces the weighted norm $\|\bullet\|_\varepsilon$ in $V$.
The discrete problem seeks $u_h\in V_h$ such that 
\begin{align}\label{def:ah-singular-perturbed}
	a_{\varepsilon,h}(u_h,v_h) 
       \coloneqq \varepsilon^2(D^2_h u_h, D^2_h v_h)_\Omega 
           + (\nabla_\pw u_h, \nabla_\pw v_h)_\Omega
          =
         (f,v_h)_\Omega
\end{align}
for any $v_h \in V_h$. 
We endow $V_h$ with the weighted norm $\|\bullet\|_{\varepsilon,h}$ induced by $a_{\varepsilon,h}$.
Throughout this discussion, all constants hidden in the notation $\lesssim$ are independent of $\varepsilon$.
For the sake of brevity, we only carry out the analysis
for $\varepsilon \leq h_{\max}$ 
since the case $h_{\max} < \varepsilon$ is simpler,
cf.~\Cref{r:s-perturbed-asymptotic} for further details.
In addition to (C1)--(C5), we will assume continuity of trial
functions as in \Cref{rem:mod-smoothing} and, for simplicity,
quasi-uniform meshes.
\begin{lemma}[singular perturbed smoothing]\label{l:perturbed-smoothing}
	Suppose (C1)--(C4), continuity $V_h\subset H^1_0(\Omega)$
	and quasi-uniformity of the mesh family under consideration,
	as well as $\varepsilon\leq h_{\mathrm{max}}$.
	Then there
	exists a linear bounded operator $J : V_h \to V$ such that any $v_h \in V_h$ satisfies
	\begin{align}\label{ineq:J-loc-singular-perturbed}
		h_{\max}^{-2}\|v_h - J v_h\|^2 + \|\nabla(v_h - J v_h)\|^2 + \varepsilon^2\|D^2_\pw (v_h - J v_h)\|^2&\nonumber\\
		\lesssim \varepsilon \sum_{F \in \Fcal} \|[\nabla v_h \cdot \nu_F]_F\|^2_F + \|\nabla_h v_h - \nabla v_h\|^2&
	\end{align}
	as well as \eqref{eq:J-orth} and \eqref{eq:J-orth-singular}.
\end{lemma}
\begin{proof}
	Let $v_h \in V_h$ be given. Following the localization argument of
	\cite[Lemma 3.5]{GallistlTian2024smai} with the $H^1_0(\Omega)$ conformity
	and the quasi-uniformity, we find
	a conforming discrete approximation $v_C \in V$ 
        over a uniformly refined subtriangulation $\hat{\Tcal}$ of $\Tcal$ 
        with maximal mesh-size $\hat{h}_{\max} \approx \varepsilon$ such that
	\begin{align*}
		h_{\max}^{-2}\|r_h\|^2 + \|\nabla r_h\|^2 + \varepsilon^2\|D^2_\pw r_h\|^2 \lesssim \varepsilon \sum_{F \in \Fcal} \|[\nabla v_h]_F \cdot \nu_F\|^2_F,
	\end{align*}
	where $r_h \coloneqq v_h - v_C$.
	To enforce the conditions \eqref{eq:J-orth} and \eqref{eq:J-orth-singular}, we construct a corrector function $w_C \in V$ as outlined in the proof of \Cref{l:smoothing} and \Cref{rem:mod-smoothing} satisfying \eqref{ineq:corrector-bound}.
	The function $J v_h \coloneqq v_C + w_C$ then satisfies \eqref{ineq:J-loc-singular-perturbed}
	as well as \eqref{eq:J-orth} and \eqref{eq:J-orth-singular}.
	The continuity of $J$ (in weighted norms) follows from the weighted trace inequality
	\begin{align*}
		\varepsilon \sum_{F \in \Fcal} \|[\nabla v_h]_F \cdot \nu_F\|^2_F \lesssim \|\nabla v_h\|^2 + \varepsilon\|\nabla v_h\|\|D^2_\pw v_h\|,
	\end{align*}
	in the regime $\varepsilon \leq h_{\max}$,
	\eqref{ineq:J-loc-singular-perturbed},
	and \eqref{e:normeq_nabla}, that is
	$\|J v_h\|_\varepsilon \lesssim \|v_h\|_{\varepsilon,h}$.
\end{proof}
As consequence of \Cref{l:perturbed-smoothing}, we obtain strengthened
versions of previous results.
First, any $v \in V$ satisfies
\begin{align}\label{ineq:qi-singular-perturbed}
	\varepsilon\|D^2(v - J Q_h v)\| \lesssim \varepsilon\|(1 - \Pi_0) D^2 v\| + \|\nabla(v - Q_h v)\| +  \|\nabla v - \nabla_h Q_h v\|
\end{align}
from \eqref{ineq:J-loc-singular-perturbed}, the weighted trace inequality, (C5), and the triangle inequality.
Second,
any $v_h,w_h \in V_h$ and $p \in P_2(\Tcal)$ satisfy
\begin{align}\label{ineq:d-cons-singular-perturbed}
	\varepsilon^2|(D_\pw^2 p, D_h^2 v_h - D^2 J v_h)_{\Omega}| \lesssim 
	\varepsilon\|D^2_\pw (p - J w_h)\| \|v_h\|_{\varepsilon,h}.
\end{align}
To prove this, we apply the trace inequality to the bound prior 
to \eqref{ineq:p2-cons-pr} to obtain
\begin{align*}
	&\varepsilon^2|(D_\pw^2 p, D_h^2 v_h - D^2 J v_h)_{\Omega}|\\ 
	&\qquad\lesssim \varepsilon^2(\|h^{-1/2} A\| + \|A\|^{1/2}\|D_\pw A\|^{1/2}) (\|h^{-1/2}B\| + \|B\|^{1/2}\|D_\pw B\|^{1/2})
\end{align*}
with the abbreviation $A \coloneqq \nabla_h v_h - \nabla J v_h$ and $B \coloneqq D^2_\pw(p - J w_h)$. 
Since $B$ is a piecewise polynomial function in $\hat{\Tcal}$, a triangulation with maximal mesh-size $\hat{h}_{\max} \approx \varepsilon$,
an inverse estimate shows $\|h^{-1/2}B\| + \|B\|^{1/2}\|D_\pw B\|^{1/2} \lesssim \varepsilon^{-1/2}\|B\|$.
From this, the foregoing displayed formula, the Young inequality
$\|A\|^{1/2}\|D_\pw A\|^{1/2} \leq \varepsilon^{-1/2}\|A\| + \varepsilon^{1/2}\|D_\pw A\|$,
and the continuity of $J$,
we infer \eqref{ineq:d-cons-singular-perturbed}.
Finally, for $f \in L^2(\Omega)$, we have the bound
\begin{align}\label{ineq:osc-singular-perturbed}
	\operatorname{osc}(f,\Tcal) \lesssim \|h(1 - \Pi_0) f\|
\end{align}
on the data oscillation by a Poincar\'e inequality with \eqref{eq:J-orth} and the continuity of $J$.
In the following, we abbreviate, for any $v \in V$,
\begin{align*}
	\mathcal{R}(v) \coloneqq \|(1 - \Pi_0) \nabla v\| + \|\nabla(v - Q_h v)\| +  \|\nabla v - \nabla_h Q_h v\|.
\end{align*}
\begin{corollary}[a~priori singular perturbed]\label{thm:a-priori-singular}
	Suppose quasi-uniform meshes,
	(C1)--(C5),
	$V_h\subseteq H^1_0(\Omega)$,
	$f \in L^2(\Omega)$, and $\varepsilon \leq h_{\max}$.
    The solutions $u$ to \eqref{def:a-singular-perturbed}
    and $u_h$ to \eqref{def:ah-singular-perturbed} satisfy
	\begin{align*}
		&\varepsilon\|\sigma - \sigma_h\| 
		+ \|\nabla (u - u_h)\| 
		\lesssim \varepsilon\|(1 - \Pi_0) \sigma\| 
		+ \mathcal{R}(u)
		+ \|h(1-\Pi_0) f\|.
	\end{align*}
\end{corollary}
\begin{proof}
   	Let $e_h = Q_h u - u_h$.
    We proceed along the lines of the proof of \Cref{thm:a-priori},
    but replace \Cref{lem:dis-cons} by \eqref{ineq:d-cons-singular-perturbed} and
    \Cref{l:C1qi} by \eqref{ineq:qi-singular-perturbed}.
    This leads to
	\begin{align}\label{e:split_sing_proof}
		\begin{aligned}
			a_{\varepsilon,h}(Q_h u, e_h) - a_\varepsilon(u, J e_h) 
			&\leq C\varepsilon(\|(1 - \Pi_0) \sigma\| + \mathcal{R}(u))\|e_h\|_{\varepsilon,h}\\
			&\qquad + (\nabla Q_h u, \nabla e_h)_\Omega - (\nabla u, \nabla J e_h)_\Omega
		\end{aligned}
	\end{align}
	with a generic positive constant $C$.
    From the orthogonality \eqref{e:gradcons_singpert} and continuity of $J$, we infer
    for the last two terms on the right-hand side of 
    \eqref{e:split_sing_proof} that 
	\begin{align*}
		(\nabla Q_h u, \nabla e_h)_\Omega - (\nabla u, \nabla J e_h)_\Omega &= (\nabla(Q_h u - u), \nabla e_h)_\Omega + (\nabla u, \nabla (e_h - J e_h))_\Omega\\
		&\lesssim (\|\nabla(u - Q_h u)\| + \|(1 - \Pi_0) \nabla u\|)\|e_h\|_{\varepsilon,h}.
	\end{align*}
	The combination of this
        with \eqref{e:split_sing_proof}, \eqref{eq:pr-a-priori-split}, and \eqref{ineq:osc-singular-perturbed} results in
	\begin{align*}
		\|e_h\|_{\varepsilon,h} \lesssim \varepsilon\|(1 - \Pi_0) \sigma\|
		+ \mathcal{R}(u) + \|h(1-\Pi_0) f\|.
	\end{align*}
	The assertion ensues from this and a triangle inequality.
\end{proof}
To derive convergence rates from the a~priori error estimate of 
\Cref{thm:a-priori-singular}, we need explicit knowledge on
the degrees of freedom of the discrete trial space $V_h$.
We refer to \Cref{cor:DKT} for an exemplary application to DKT elements.
If $h_{\max} \leq \varepsilon$, then we obtain the following a~priori error estimate.
\begin{remark}[$h_{\max} \leq \varepsilon$]\label{r:s-perturbed-asymptotic}
If $h_{\max} \leq \varepsilon$, then a modification of the smoothing operator $J$ in \Cref{l:smoothing} is not necessary as the stability
	\begin{align*}
		\|\nabla J v_h\| \lesssim \|\nabla_\pw v_h\| 
		+ h_{\max}\|D^2_\pw v_h\| \leq \|\nabla_\pw v_h\| + \varepsilon\|D_\pw^2 v_h\|
	\end{align*}
	holds and thus $\|J v_h\|_\varepsilon \lesssim \|v_h\|_{\varepsilon,h}$ for any $v_h \in V_h$.
	Following the arguments presented in the proof of \Cref{thm:a-priori-singular},
	we deduce, under the assumptions (C1)--(C5),
	continuity of trial functions, and $f \in L^2(\Omega)$, that
	\begin{align*}
		&\varepsilon\|\sigma - \sigma_h\| + \|\nabla (u - u_h)\| \lesssim \varepsilon\|(1 - \Pi_0) \sigma\|\\
		&\qquad + \|(1 - \Pi_0) \nabla u\| + \|\nabla(u - Q_h u)\| + \varepsilon^{-1}\|h^2(1-\Pi_0)f\|.
	\end{align*}
\end{remark}

\begin{remark}
 The $H^1_0(\Omega)$ conformity is a sufficient criterion for
 consistency of the scheme in the formal limit $\varepsilon=0$,
 see \cite{NilsenTaiWinther2001}.
 	For the more general case of non quasi-uniform meshes
        with a certain grading,
        a construction in the spirit of
 	\cite[Lemma 3.5]{GallistlTian2024smai},
 	would require 
        a subtriangulation $\hat{\Tcal}$ of $\Tcal$ so
 	that the local mesh-size $\hat{h}$ satisfies
 	$\hat{h} \lesssim \varepsilon$ a.e.~in $\Omega$ and $h/\hat{h} \lesssim 1$ wherever $h \leq \varepsilon$.
 	Such construction remains technically challenging.
\end{remark}

\section{A~posteriori error analysis}\label{sec:pr-a-post}

This section is devoted to the proof of 
the a~posteriori error bound of \Cref{thm:a-post}.
With the abbreviation $\bar\sigma_h=\Pi_0\sigma_h$, we define the error estimator
\begin{align}\label{e:mudef}
	\mu^2 \coloneqq \|&h^2 (f - \div_\pw \div_\pw \sigma_h)\|^2 
      + \|\operatorname{skw}\,\sigma_h\|^2 + \|\sigma_h - \bar\sigma_h\|^2 \nonumber\\
	&+ \sum_{\substack{F \in \Fcal\\ F\subset\partial\Omega}}
       h_F\|[\sigma_h^{\mathit{tang}}]_F\|_F^2 
        + \sum_{\substack{F \in \Fcal \\ F\not\subset\partial\Omega}}
	\Big(h_F\|[\sigma_h]_F\|_F^2 
	+ h_F^3\|[\div_\pw \sigma_h \cdot \nu_F]_F\|_F^2\Big),
\end{align}
where $\sigma_h^{\mathit{tang}} \coloneqq \sigma_h(I_{n\times n}-\nu_F\nu_F^\top)$ denotes
the tangential component of $\sigma_h$.
Here and throughout this section, $\operatorname{skw} M$ denotes the skew-symmetric part
of a matrix $M$, while $\operatorname{sym} M$ is its symmetric part.

\begin{proof}[Proof of \Cref{thm:a-post}]
	The proof departs from the split
	\begin{align}\label{eq:pr-a-post-err-split}
		&\|\sigma - \sigma_h\|^2 
        = \inf_{\psi \in V}\|D^2 \psi - \sigma_h\|^2 
       +
       \left( \sup_{\psi \in V \setminus \{0\}} \int_\Omega (\sigma - \sigma_h) : D^2 \psi /\|D^2 \psi\|
       \right)^2
	\end{align}
	of the error $\|\sigma - \sigma_h\|$ into a nonconforming and consistency part
        \cite{CarstensenGallistlHu2013}.
	The nonconforming error $\inf_{\psi \in V}\|D^2 \psi - \sigma_h\|$ is controlled by \eqref{ineq:v_h-v_c}, the Poincar\'e inequality with (C1), and (C3), namely
	\begin{align}\label{ineq:pr-a-post-nc-error}
		\inf_{\psi \in V}\|D^2 \psi - \sigma_h\| \leq \|D^2 J u_h - \sigma_h\|^2 \lesssim \|(1 - \Pi_0) \sigma_h\|^2 
		+ \sum_{F \in \Fcal} h_F\|[\sigma_h^{\mathit{tang}}]_F \|_F^2.
	\end{align}
        We proceed with bounding the second term on the right-hand side
        of \eqref{eq:pr-a-post-err-split}.
	Any $\psi \in V$ satisfies
	\begin{align}\label{eq:pr-a-post-err-split-cons}
		\int_\Omega &(\sigma-\sigma_h) : D^2 \psi 
       	= T_1 + T_2 + T_3
	\end{align}
        with the terms
	\begin{align*}
        \begin{aligned}
        & T_1\coloneqq \int_\Omega (\sigma - \sigma_h) : (D^2 \psi - D^2 J Q_h \psi), \\
        & T_2\coloneqq \int_\Omega \left( \sigma : D^2 J Q_h \psi 
                                  - \sigma_h : D_h^2 Q_h \psi 
                           \right)
        ,\quad 
        T_3 \coloneqq	- \int_\Omega \sigma_h : (D^2 J Q_h \psi - D_h^2 Q_h \psi).
        \end{aligned}
	\end{align*}
	Two piecewise integrations by parts and \eqref{def:continuous-problem} show
	\begin{align*}
		&T_1 = \int_\Omega (f - \div_\pw \div_\pw \sigma_h)(\psi - JQ_h\psi)\\
		&~- \sum_{F \in \Fcal} 
               \int_F \left(\nabla (\psi - JQ_h\psi) \cdot [\sigma_h \nu_F]_F - (\psi - JQ_h\psi)\, [\div_\pw \sigma_h \cdot \nu_F]_F
                     \right).
	\end{align*}
	Standard techniques in a~posteriori error estimation with \Cref{l:C1qi} lead to
	\begin{align}\label{ineq:t1}
		T_1 \lesssim \mu \|D^2 \psi\|. 
	\end{align}
	The solution properties of $u$ and $u_h$ in \eqref{def:continuous-problem} and \eqref{def:discrete-problem}, \Cref{lem:osc}, and (C5) show
	\begin{align}\label{ineq:t23}
		T_2= \int_\Omega f(JQ_h \psi - Q_h \psi)  
              \lesssim \|h^2(1 - \Pi_0) f\| \|D^2 \psi\|.
	\end{align}
	From the best approximation property of $\Pi_0$ and the triangle inequality, we infer
	\begin{align*}
		\|h^2(1 - \Pi_0) f\| \leq \|h^2 (f - \div_\pw \div_\pw \sigma_h)\| + \|h^2(1 - \Pi_0) \div_\pw \div_\pw \sigma_h\|.
	\end{align*}
	Since $\div_\pw \div_\pw \bar\sigma_h \equiv 0$, the inverse estimate shows $\|h^2 \div_\pw \div_\pw \sigma_h\| \lesssim \|\sigma - \bar\sigma_h\|$. Therefore, $\|h^2(1 - \Pi_0) f\|$ is dominated by the error estimator $\mu$.  
	For the final term $T_3$, we employ 
	\Cref{lem:dis-cons} to infer, for any piecewise
        quadratic $p$, that
	\begin{align*}
		|(D^2_\pw p, D^2 J Q_h \psi - D_h^2 Q_h \psi)| \lesssim \|D^2_\pw(p - J u_h)\|\|D^2 \psi\|.
	\end{align*}
        Since $D^2_\pw P_2(\Tcal) = P_0(\Tcal;\mathbb{S})$
        (cf.\ the lines preceding \eqref{eq:Galerkin-err}), this and the triangle inequality imply
	\begin{align}\label{ineq:t4}
        \begin{aligned}
		T_3 &\lesssim \min_{\Phi \in P_0(\Tcal;\mathbb{S})} \big(\|\sigma_h - \Phi\| + \|\Phi - D^2 J u_h\|\big) \|D^2 \psi\|\\
		&\lesssim \Big(\min_{\Phi \in P_0(\Tcal;\mathbb{S})} \|\sigma_h - \Phi\| + \|\sigma_h - D^2 J u_h\|\Big) \|D^2 \psi\|.
        \end{aligned}
	\end{align}
	We note that 
        $$
         \min_{\Phi \in P_0(\Tcal;\mathbb{S})} \|\sigma_h - \Phi\|^2 
         =
         \|\mathrm{skw}\, \sigma_h\|^2 
           + \|(1 - \Pi_0) \mathrm{sym} \,\sigma_h\|^2 
         \leq
         \|\mathrm{skw}\, \sigma_h\|^2 + \|(1 - \Pi_0) \sigma_h\|^2
         .
        $$
         Thus, the assertion follows from
	\eqref{eq:pr-a-post-err-split}--\eqref{ineq:t4}.
\end{proof}

\begin{remark}[efficiency]
	The error estimator $\mu$ in \eqref{e:mudef} is, 
	up to the best approximation error $\|(1 - \Pi_0) \sigma\|$ by piecewise constants,
	efficient in the sense that
	\begin{align}\label{ineq:eff}
		\mu \lesssim \|(1 - \Pi_0) \sigma\| 
                + \|\sigma - \sigma_h\| 
               + \|h^2(1-\Pi_0)f\|.
	\end{align}
	In fact, the symmetry of $\sigma$ and a triangle inequality imply
	\begin{align*}
		\|\mathrm{skw}\,\sigma_h\| 
                + \|(1 - \Pi_0) \sigma_h\| \leq \|\sigma - \sigma_h\| 
               + 2\|(1 - \Pi_0) \sigma\|. 
	\end{align*}
	The efficiency of the remaining terms follow from standard bubble function techniques \cite{Verfuerth2013}, cf.~also \cite{BeiraodaVeigaNiiranenStenberg2007,HuShi2009,GallistlTian2024smai,LiangTran2026} for further details.
\end{remark}

\begin{remark}[piecewise constant post-processing]\label{rem:simplified_est}
	By triangle, discrete trace, and inverse inequalities,
        the error estimator $\mu$ from \eqref{e:mudef} is equivalent to
	\begin{align*}
		\widetilde{\mu}^2 &\coloneqq \|h^2 f\|^2 
               + \|\mathrm{skw}\,\bar\sigma_h\|^2 + \|\sigma_h-\bar\sigma_h\|^2
               \\
		&\quad+ 
                 \sum_{\substack{F \in \mathcal{F}\\ F\subset\partial\Omega}}
                          h_F\|[\bar\sigma_h^{tang}]_F\|_F^2 
               + \sum_{\substack{F \in \mathcal{F}\\ F\not \subset\partial\Omega}}
		h_F\|[\bar \sigma_h]_F\|_F^2.
	\end{align*}
    Considering \eqref{ineq:eff} and \Cref{thm:a-priori}, $\widetilde{\mu}$ is a reliable and efficient error estimator for $\|\sigma - \bar{\sigma}_h\|$, i.e.,
	\begin{align*}
		\|\sigma - \bar{\sigma}_h\| \lesssim \widetilde{\mu} \lesssim \|(1 - \Pi_0) \sigma\| + \|h^2(1-\Pi_0)f\| \leq \|\sigma - \bar{\sigma}_h\| + \|h^2(1-\Pi_0)f\|.
	\end{align*}
\end{remark}

\section{Discrete Kirchhoff triangle and its generalization to three space dimensions}
\label{s:dkt}

We apply the results from prior sections to the DKT
setting. The method for $n=2$ was proposed in \cite{BatozBatheHo1980}
and described and analyzed in \cite{Braess2007,Bartels2015}.
Our description covers the case $n=3$
as well and leads to a very simple low-order scheme
in three space dimensions, which extends the known element
by maintaining the underlying idea. The error analysis
simultaneously covers the cases $n\in\{2,3\}$.

Given a simplex $T \subset \mathbb R^n$, $n \in \{2,3\}$, 
with the vertices $z_j$ for $j = 1, \dots, n+1$ and the midpoints $a_{jk\ell} \coloneqq (z_j + z_k + z_\ell)/3$ for $1 \leq j < k < \ell \leq n+1$, 
recall from \cite[Theorem 2.2.8]{Ciarlet2002} that any cubic polynomial $p$
is uniquely determined by the nodal values and the values of its first derivatives at the vertices $z_j$, $1 \leq j \leq n+1$ as well as the nodal values at $a_{jk\ell}$.
For each triplet $j,k,\ell$, let
\begin{align*}
	\psi_{jk\ell}(p) \coloneqq 6p(a_{jk\ell}) - \sum\nolimits_{m \in \{j,k,\ell\}} (2p(z_m) - \nabla p(z_m) \cdot (z_m - a_{jk\ell})).
\end{align*}
We define the reduced set
\begin{align*}
	P_3^-(T) \coloneqq \{p \in P_3(T) : \psi_{jk\ell}(p) = 0 \text{ for any } 1 \leq j < k < \ell \leq n+1\}
\end{align*}
of cubic polynomials by removing the
degrees of freedom associated with the nodal evaluation at $a_{jk\ell}$, $1 \leq j < k < \ell \leq n+1$.
Thus, any $p \in P_3^-(T)$ is uniquely defined by prescribing the nodal values and first derivatives at the $n+1$ vertices and so, $\dim P_3^-(T) = (n+1)^2$ \cite[Theorem 2.2.9]{Ciarlet2002}.
Furthermore, let
\begin{align*}
	\Theta_h(T) \coloneqq 
       \{\theta_h \in P_2(T)^n : \theta_h \cdot \nu_F \in P_1(F) ~\text{for any face }
          F \text { of } T\}.
\end{align*}
Recall that $\theta_h \in P_2(T)^n$ is uniquely defined by the values at all 
vertices and midpoints of all edges of $T$.
The $n$ values at the midpoint of an edge $E$ are fixed in $n-1$ 
linear independent directions from the side conditions 
$\theta_h \cdot \nu_F \in P_1(F)$ on $n-1$ adjacent faces.
Thus, only the tangential direction $t_E$, $\theta_h(m_E)\cdot t_E$,
remains as a degree of freedom in the midpoint $m_E$ of $E$.
Hence, any $\theta_h \in \Theta_h(T)$ is uniquely defined by 
the values at the vertices and midpoints of all edges along 
the tangential direction $t_E$.
This leads to $\dim \Theta_h(T) = n(n+1) + n(n+1)/2 = 3n(n+1)/2$.

\begin{remark}[unique extension by edge values]\label{rem:unique-extension}
	If any two functions $v_h, w_h \in P_3^-(T)$ on a simplex $T$
	coincide along all edges of $T$, then $v_h = w_h$.
	In fact, given a vertex $z$, then $v_h(z) = w_h(z)$ 
	and $\nabla v_h(z) \cdot t_E = \nabla w_h(z) \cdot t_E$ for any 
	edge $E$ containing $z$.
	The $n$ tangential directions $t_E$ of these edges are linear independent,
	whence the gradients of $v_h$ and $w_h$ coincide at $z$.
\end{remark}

The discretization utilizes the DKT element of 
\cite{BatozBatheHo1980} with the discrete space
\begin{align*}
	V_h &\coloneqq \{v \in H^1_0(\Omega) : v|_T \in P_3^-(T) \ \text{ for all } T \in \mathcal{T},\\
	&\qquad\nabla v \text{ is continuous at all vertices and 0 at all boundary vertices}\}
\end{align*}
and reconstructs the gradient in 
\begin{align*}
	\Theta_h \coloneqq \{\theta_h \in [H^1_0(\Omega)]^n : \theta_h|_T \in \Theta_h(T)\}.
\end{align*}
The discrete gradient operator $\nabla_h : V_h \to \Theta_h$ maps $v_h \in V_h$ onto $\nabla_h v_h \in \Theta_h$
with
\begin{align}
	\nabla_h v_h(z) &= \nabla v_h(z) 
	    &&\text{for any vertex } z,\label{def:discrete-gradient-1}\\
	\nabla_h v_h(m_E) \cdot t_E 
	    &= \nabla v_h(m_E) \cdot t_E 
		&&\text{for any edge } E
               \text{ with midpoint }m_E
	\label{def:discrete-gradient-2}
	.
\end{align}
We proceed by verifying conditions (C1)--(C5) for this method.
The properties (C1) and (C2) follow from the definition of $V_h$, whereas
(C3)--(C5) follow from the next results.

\begin{lemma}[(C3)--(C5) for DKT]\label{l:c4dkt}
	The DKT elements satisfy (C3)--(C5).
	The constants hidden in the notation may depend on $T$ but remain
        bounded for all $T \in \mathbb{T}$, where $\mathbb{T}$ denotes 
        a class of triangulation involving finitely many shapes.
\end{lemma}
\begin{proof}
	The stated inequalities are invariant under translation and scaling.
	Since only a finite number of different simplex shapes are involved,
	the constants remain uniformly bounded for that class of meshes.
	
	Let $v_h \in P_3^-(T)$ for some simplex $T$.
	We begin with proving (C3).
	We first establish
	$\|D^2_h v_h - D^2 v_h\|_T 	\lesssim \|(1 - \Pi_0) D^2 v_h\|_T$.
   For the proof of this estimate, we note that,
   if the left-hand side vanishes, then $v_h$ is a quadratic polynomial.
   It is straightforward to verify that then the affine vector fields 
   $\nabla_h v_h = \nabla v_h$ coincide
   by the definition of $\nabla_h$ 
   from \eqref{def:discrete-gradient-1}--\eqref{def:discrete-gradient-2}.
   Therefore, the left-hand side vanishes and, 
   by equivalence of norms in finite dimensional spaces,
   the claim follows.
   
   We furthermore claim that
   $\|(1 - \Pi_0) D^2 v_h\|_T	\lesssim \|(1 - \Pi_0) D^2_h v_h\|_T$.
	If the right-hand side of this relation vanishes,
	then $\nabla_h v_h$ is an affine vector field.
	From the assignment in 
	\eqref{def:discrete-gradient-1}--\eqref{def:discrete-gradient-2}, 
	we deduce that on an arbitrary edge $E$ of $T$,
	the tangential derivative $\partial v_h/\partial t_E$,
	which is a quadratic polynomial along $E$,
	coincides with $\nabla_h v_h \cdot t_E$.
	Thus, the tangential derivative of $v_h$ along $E$ is affine.
	This implies that $v_h$ is quadratic along $E$.
	Since $P_2(T) \subset P_3^-(T)$ 
	\cite[Theorem 2.2.9]{Ciarlet2002} and all degrees of freedom of 
	$P_2(T)$ lie on the union of edges of $T$,
	there exists a function $\tilde{v}_h \in P_2(T)$ 
	with $\tilde{v}_h|_E = v_h|_E$ for any edge $E$ of $T$.
	By \Cref{rem:unique-extension}, $v_h = \tilde{v}_h \in P_2(T)$ 
	and therefore, the left-hand side vanishes and, by equivalence of norms 
	in finite space dimensions, the claim ensues. 
	This proves (C3).
	
	For the proof of (C4), we note that the assignment
	\eqref{def:discrete-gradient-1} implies
	$(\nabla v_h - \nabla_h v_h)(z) = 0$ for any vertex $z$ of $T$.
	Therefore, constants are eliminated and (C4) follows
	from a discrete Poincar\'e inequality.
	For verifying (C5), we design an averaging operator $Q_h$ as follows.
	Given a piecewise polynomial function $w_h$,
	the nodal average $\mathcal{A}_h w_h \in V_h$ of $w_h$ 
	is uniquely defined by the nodal values
       (the $\Sigma$ with the bar represents the average)
	\begin{align*}
	\mathcal{A}_h w_h(z) 
		\coloneqq \overline{\sum_{T \in \mathcal{T}_z}} w_h|_T(z)
               \quad\text{and}\quad 
            \nabla \mathcal{A}_h w_h(z) 
            \coloneqq  \overline{\sum_{T \in \mathcal{T}_z}} \nabla w_h|_T(z)
	\end{align*}
	for any interior vertex $z$, where
	$\mathcal{T}_z$ denotes the set of of all simplices containing $z$.
	Standard averaging techniques show the bound
	\begin{align}\label{ineq:proof-interpolation-averaging}
        \begin{aligned}
	  &\sum_{j=0}^2	h_T^{2(j-2)}\|D^{j}(w_h - \mathcal{A}_h w_h)\|_T^2
	  \lesssim \sum_{\substack{F \in \mathcal{F}\\F \cap T \\\neq \emptyset}}
               \big(h_F^{-3}\|[w_h]_F\|_F^2 + h_F^{-1}\|[\nabla w_h]_F\|_F^2\big).
        \end{aligned}
	\end{align}
	Given $v \in V$, we define the quasi-interpolation as 
        $v_h = Q_h v \coloneqq \mathcal{A}_h \Pi_h v$, 
        where $\Pi_h$ denotes the $L^2$ orthogonal projection onto 
        the piecewise polynomial functions that belong to
        $P_3^-(T)$ when restricted to any simplex $T\in\Tcal$.
         The choice $w_h \coloneqq \Pi_h v$ in \eqref{ineq:proof-interpolation-averaging},
	$[w_h]_F = [v - w_h]_F$, $[\nabla w_h]_F = [\nabla(v - w_h)]_F$ for any $F \in \mathcal{F}$, and the trace inequality imply
	\begin{align*}
	 	\sum_{j=0}^2 \|h_T^{j-2}D^{j}(v - Q_h v)\|_{L^2(T)}
	 	\lesssim
		\sum_{j=0}^2 \|h_T^{j-2}D_\pw^{j}(v - \Pi_h v)\|_{L^2(\omega_T)}.
	\end{align*}
	Since $P_2(K) \subset P_3^-(K)$ for any simplex $K$,
	this and the Poincar\'e inequality conclude the proof of the 
	error bound in (C5).
\end{proof}

Since the DKT element satisfies (C1)--(C5) and
$H^1_0(\Omega)$ conformity, we obtain the following
error bounds.
Recall that $\sigma_h = D^2_h u_h = D \nabla_h u_h$.

\begin{corollary}[DKT error bounds]\label{cor:DKT}
        Let $u$ denote the solution to \eqref{def:continuous-problem}
        and $u_h$ denote the discrete solution to
        $\eqref{def:discrete-problem}$ discretized with the Discrete
        Kirchhoff method described in this section.
	They satisfy
	\begin{align*}
		\|\sigma - \sigma_h\| \lesssim \|(1 - \Pi_0) \sigma\| + \osc(f,\Tcal).
	\end{align*}
	If the elliptic regularity \eqref{ineq:elliptic-regularity} is satisfied, then
	\begin{align*}
		\|\nabla_\pw(u - u_h)\| \lesssim h_{\max}^\delta(\|(1 - \Pi_0) \sigma\| + \widetilde{\osc}(f,\Tcal)).
	\end{align*}
        Furthermore, the following a~posteriori error bound holds
	\begin{align*}
		\|\sigma - \sigma_h\| \lesssim \eta
               \lesssim \|\sigma-\sigma_h\| + \|(1-\Pi_0)\sigma\|
                       +\|h^2(1-\Pi_0)f\|
	\end{align*}
	with the error estimator
	\begin{align*}
		\eta^2 &\coloneqq \|h^2 f\|^2 + \|\operatorname{skw}\bar\sigma_h\|^2
		+ \|\sigma_h - \bar\sigma_h\|^2
		+ \sum_{\substack{F\in\mathcal{F} 
                        \\ F\not\subset\partial\Omega}}
		 h_F\|[\bar\sigma_h \nu_F]_F\|_F^2 
		.
	\end{align*}
	For the singularly perturbed problem \eqref{def:a-singular-perturbed},
	the solutions $u$ to \eqref{def:a-singular-perturbed} and
        $u_h$ to \eqref{def:ah-singular-perturbed}
        satisfy, for quasi-uniform meshes,
        the error bound from \Cref{thm:a-priori-singular}.
	If $\Omega \subset \mathbb{R}^2$ is convex, then
	\begin{align*}
		\varepsilon\|\sigma - \sigma_h\| + \|\nabla(u - u_h)\| \lesssim h_{\max}^{1/2}\|f\|.
	\end{align*}

\end{corollary}
\begin{proof}
	The first asserted a~priori error bounds follow
        from \Cref{thm:a-priori} and \Cref{t:Hs_estimate}.
        Since $\nabla_h u_h \in H^1_0(\Omega)$ is continuous,
	the tangential jump of $\sigma_h$ vanishes along any face $F \in \Fcal$.
	Hence, \Cref{thm:a-post} and \Cref{rem:simplified_est} imply the asserted a~posteriori bound.
	For the singular perturbed problem \eqref{def:a-singular-perturbed},
	the trace inequality in \eqref{ineq:proof-interpolation-averaging} and interpolation properties of $\Pi_h$ imply
	\begin{align*}
		\|\nabla(\Pi_h v - \mathcal{A} \Pi_h v)\|_T^2 \lesssim h_T\|\nabla(v - \Pi_h v)\|_{\omega_T}\|D^2_\pw(v - \Pi_h v)\|_{\omega_T}
	\end{align*}
	for any $v \in V$ and $T \in \Tcal$, similar to \cite[Ineq.~(4.4)]{NilsenTaiWinther2001}. The sum of this over $T\in \Tcal$, a triangle inequality, and the stability $\|\nabla_\pw \Pi_h v\| \lesssim \|\nabla v\|$ provide
	\begin{align}\label{ineq:err-sing-perturbed}
		\|\nabla(v - Q_h v)\|^2 \lesssim h_{\max}\|\nabla v\|\|D^2 v\|.
	\end{align}
	To bound $\|\nabla v - \nabla_h Q_h v\|$, we observe
	that $\nabla_h Q_h v$ and $\nabla Q_h v$ coincide
	in all degrees of freedom of $[P_2(\Tcal)]^n$, except for
	normal directions of point evaluations in edge midpoints,
	where $\nabla_h Q_h v$ behaves as $I_1 \nabla Q_h v$ with the standard $P_1$ nodal interpolation $I_1$.
	Therefore, equivalence of norms in finite dimensional spaces implies
	\begin{align*}
		\|\nabla Q_h v - \nabla_h Q_h v\| \lesssim \|(1 - I_1) \nabla Q_h v\|.
	\end{align*}
	From this and \eqref{ineq:err-sing-perturbed},
	we derive, with standard interpolation arguments, that
	\begin{align*}
		\|\nabla v - \nabla_h Q_h v\| \lesssim h_{\max}\|\nabla v\|\|D^2 v\|.
	\end{align*}
	With this and \eqref{ineq:err-sing-perturbed},
	it is straightforward to extend the proof of \cite[Theorem 5.2]{NilsenTaiWinther2001} to the DKT elements and we conclude the displayed convergence rates in planar convex domains.
\end{proof}
To simplify the implementation, we can modify the discrete right-hand side with the nodal interpolation $I_1$ as suggested in \cite{Bartels2015}. The modified discrete problem seeks $\tilde{u}_h \in V_h$ such that
\begin{align}\label{e:dkt-nodal-int}
	a_h(\tilde{u}_h, v_h) = (f,v_h)_{\Omega} \quad\text{for any } v_h \in V_h.
\end{align}
It is straightforward to verify the following a~priori error estimate.
\begin{corollary}[DKT with nodal interpolation]\label{cor:dkt-nodal}
	The discrete solution $\tilde{u}_h$ to \eqref{e:dkt-nodal-int} 
        with $\tilde \sigma_h = D^2_h u_h$ satisfies
	\begin{align*}
		\|\sigma - \tilde \sigma_h\| 
            \lesssim \|(1-\Pi_0)\sigma\| + \mathrm{osc}(f,\Tcal) + \sup_{v_h\in V_h \setminus \{0\}} 
		\langle f,v_h-I_1 v_h \rangle/\|v_h\|_h. 
	\end{align*}
\end{corollary}
\begin{proof}
	The proof follows the arguments of this section and is omitted for the sake of brevity.
\end{proof}
By approximation property of the nodal interpolation $I_1$, the consistency error $\langle f, v_h - I_1 v_h \rangle$ is of second order for $f \in L^2(\Omega)$ and of first order for $f \in H^{-1}(\Omega)$.

\section{Discrete stream functions for two-dimensional Stokes elements}
\label{s:stream}

In this section, we draw a connection between the classical
two-dimensional DKT element and the Bernardi--Raugel discretization
of the Stokes system.
This links DKT-like elements to the modified schemes initiated by
\cite{Linke2014}
and later developed in
\cite{LinkeMatthiesTobiska2016,LedererLinkeMerdonSchoeberl2017,
      JohnLinkeMerdonNeilanRebholz2017}
such that \Cref{thm:a-priori} offers an alternative error bound.
The spaces in the well known Stokes system are the velocity
space $W \coloneqq [H^1_0(\Omega)]^2$ of vector-valued $H^1$ functions
with homogeneous boundary conditions, and the pressure space
$Q \coloneqq L^2_0(\Omega)$ of $L^2$ functions with vanishing global average.
Given $f\in [L^2(\Omega)]^2$,
the Stokes system seeks $(w,p)\in W\times Q$ such that
\begin{align}\label{e:stokes}
\begin{aligned}
(Dw,Dv)_\Omega
 - (p,\div v)_\Omega
  &= (f,v)_\Omega 
  &&\text{for all } v\in W,
\\
-(\div w,q)_\Omega 
  &=  0
  &&\text{for all } q\in Q.
\end{aligned}
\end{align}
The Bernardi--Raugel discretization is based on conforming
standard first-order finite elements for the discretization of $W$
that are enriched by quadratic edge bubbles pointing in normal
direction for each interior edge, that is
\begin{equation*}
W_h
  =
 \{v\in W : v|_T \text{ affine for any }T\in\mathcal T\}
\oplus
\operatorname{span}\{ b_F \nu_F : F\in\mathcal F, F\not\subset\partial\Omega\},
\end{equation*}
where $b_F$ denotes the usual quadratic edge bubble for
an interior edge $F$. Together with the choice
$Q_h \coloneqq P_0(\mathcal T)\cap L^2_0(\Omega)$ of piecewise constants
with vanishing global mean, this yields a stable discretization
for \eqref{e:stokes}, see \cite{BoffiBrezziFortin2013}.
In order to make the error $w-w_h$ independent of the pressure
error in the system, 
\cite{Linke2014,LinkeMatthiesTobiska2016,
      JohnLinkeMerdonNeilanRebholz2017}
proposed a modification of the discrete right-hand side
involving a divergence-conforming reconstruction of discretely
divergence-free fields. This makes the discrete velocity
vaiable $w_h$ in the discrete system blind against shifts of
$f$ by gradients.
In the present situation, the standard Raviart--Thomas interpolation
\cite{BoffiBrezziFortin2013}, denoted by $I_{\mathit RT}$,
maps $W_h$ to the space $RT_0(\mathcal T)$ of lowest-order
Raviart--Thomas elements and satisfies
$\Pi_0\div w_h = \div I_{\mathit RT} w_h$ for any
$w_h\in W_h$.
The modified discrete system then seeks 
$(w_h,p_h)\in W_h\times Q_h$ such that
\begin{align}\label{e:stokes_discrete}
\begin{aligned}
(Dw_h,Dv_h)_\Omega
 - (p_h,\div v_h)_\Omega
  &= (f,I_{\mathit{RT}}v_h)_\Omega 
  &&\text{for all } v_h\in W_h,
\\
-(\div w_h,q_h)_\Omega 
  &=  0
  &&\text{for all } q_h\in Q_h.
\end{aligned}
\end{align}
Following the lines of
\cite{LinkeMatthiesTobiska2016} yields optimal first-order
convergence of the error $w-w_h$ in the $H^1$ norm, provided
$w$ belongs to $H^2(\Omega)$.

In order to show an alternative error bound, we establish that
the DKT element is equivalent to the Bernardi--Raugel pair.
The first observation is that $W_h$ equals 
the space $\Theta_h$ from
the DKT element defined in \Cref{s:dkt}, rotated by $\pi/2$
so that normal directions are mapped to tangential directions and
vice versa.
Upon defining the vector Curl of a scalar function $\varphi$
and the discrete Curl of a discrete function $\varphi_h\in V_h$ by
$$
 \operatorname{Curl} \varphi_h 
 = \begin{pmatrix} -\partial_2 \varphi_h \\ \partial_1 \varphi_h\end{pmatrix}
 = \begin{pmatrix}0&-1\\1&0\end{pmatrix}\nabla \varphi_h
\quad\text{and}\quad
 \operatorname{Curl}_h \varphi_h 
 = \begin{pmatrix}0&-1\\1&0\end{pmatrix}\nabla_h \varphi_h
,
$$
we claim that,
provided the two-dimensional
domain is simply-connected,
an element $\psi_h\in W_h$ satisfies
$\psi_h=\operatorname{Curl}_h \varphi_h$
for some $\varphi_h\in V_h$, the DKT space of \Cref{s:dkt},
if and only if $\int_T \operatorname{div} \psi_h=0$ 
for all elements $T$.
The ``only if'' implication is verified via the integration
by parts formula
\begin{align*}
	\int_T \operatorname{div} \operatorname{Curl}_h \varphi_h = \int_{\partial T} \nabla_h \varphi_h \cdot t = \int_{\partial T} \nabla \varphi_h \cdot t = 0,
\end{align*}
with a unit tangent $t$ of $\partial T$,
where we utilize the observation $\nabla_h \varphi_h \cdot t = \nabla \varphi_h \cdot t$ 
on $\partial T$ in the proof of \Cref{l:c4dkt}.
The ``if'' part follows from a dimension
count with the Euler formula on simply-connected domains
\cite{ErnGuermond2004}.
\Cref{f:commuting} illustrates the commuting discrete
relations.

If the domain $\Omega$ is assumed simply connected, it is
well known that there exists a stream function $u\in V$
such that $w=\operatorname{Curl} u$ and 
$\Delta^2 u = -\operatorname{rot} f$,
where as usual we denote by 
$\operatorname{rot} f=\partial_1f_2-\partial_2 f_1$
the scalar rotation in the plane.
The above discussion implies that an analogous relation holds
in the discrete system, namely there exists $u_h\in V_h$
such that $w_h = \operatorname{Curl}_h u_h$.
System \eqref{e:stokes_discrete} and the observation
that $I_{\mathit{RT}}\operatorname{Curl}_h$
equals $\operatorname{Curl} I_1$ for elements of $V_h$,
show that $u_h$ computed from $w_h$ from \eqref{e:stokes_discrete}
solves
\begin{equation*}
a_h(u_h, v_h) = (f,\operatorname{Curl}I_1 v_h) 
\quad\text{for all }v_h \in V_h.
\end{equation*}
This is the DKT system with right-hand side $-\operatorname{rot} f$
combined with the discrete operator $I_1$ on the right-hand
side.
We mention that the use of the nodal interpolation on the
right-hand side also simplifies the implementation in that
the basis of $V_h$ is not needed any more.
\Cref{cor:dkt-nodal} then yields:

\begin{corollary}[Bernardi-Raugel error bounds]\label{c:Bernardi-Raugel}
Assume the two-dimensional domain $\Omega$ is simply-connected
and $f\in [L^2(\Omega)]^2$. The error between the solution $w$
to \eqref{e:stokes} and the discrete solution $w_h$ to
\eqref{e:stokes_discrete} satisfies
\begin{equation*}
\| D(w-w_h) \|
\lesssim
\| (1-\Pi_0) Dw \|
+
\osc(\operatorname{rot}f,\mathcal T)
+
\sup_{v_h\in V_h \setminus \{0\}} 
  \langle \operatorname{rot} f,v_h-I_1 v_h \rangle/\|v_h\|_h.
\end{equation*}
\end{corollary}
Note that, for $w \in H^2(\Omega)$, we obtain optimal first-order convergence as in \cite{LinkeMatthiesTobiska2016}. 
Theoretically, we can provide a reconstruction operator that avoids the consistency error in \Cref{c:Bernardi-Raugel}. Note that the mapping $\nabla v_h \mapsto \nabla_h v_h$ for $v_h \in P_3^-(T)$ is injective \cite{Braess2007} for any $T \in \Tcal$. Therefore, there exists a (locally computable) operator $R : W_h \to \operatorname{Curl} V_h$ such that any $v_h \in V_h$ satisfies $\operatorname{Curl} v_h = R \operatorname{Curl}_h v_h$.
Then $u_h = \operatorname{Curl}_h w_h$ for the discrete solution $w_h$ to the Stokes system \eqref{e:stokes_discrete} with the modified right-hand side $(f,R v_h)_\Omega$ instead of $(f,I_{\mathit{RT}}v_h)_\Omega$ solves the DKT system
\begin{align*}
	a_h(u_h,v_h) = \langle \operatorname{rot} f, v_h \rangle \quad\text{for all } v_h \in V_h.
\end{align*}

\begin{figure}
\begin{center}
\begin{tikzpicture}[node distance=2.0cm]
\node (A1) {$V_h$};
\node (A2)[right of=A1]{$W_h$};
\node (A3)[right of=A2]{$Q_h$};
\node (B1)[below of=A1]{$S^1_0(\mathcal T)$};
\node (B2)[right of=B1]{$RT_0$};
\node (B3)[right of=B2]{$Q_h$};
\draw[transform canvas={yshift=0.5ex},->] (A1) --(A2) node[above,midway] {\footnotesize $\operatorname{Curl}_h$};
\draw[transform canvas={yshift=0.5ex},->] (A2) --(A3) node[above,midway] {\footnotesize $\Pi\circ\div$};
\draw[transform canvas={xshift=0.5ex},->] (B1)--(B2) node[above,midway] {\footnotesize $\operatorname{Curl}$};
\draw[transform canvas={xshift=0.5ex},->] (B2)--(B3) node[above,midway] {\footnotesize $\operatorname{div}$};
\draw[transform canvas={xshift=0.5ex},->] (A1)--(B1) node[left,midway] {\footnotesize $I_1$};
\draw[transform canvas={xshift=0.5ex},->] (A2)--(B2) node[left,midway] {\footnotesize $I_{RT}$};
\draw[transform canvas={xshift=0.5ex},->] (A3)--(B3) node[left,midway] {\footnotesize $id$};
\end{tikzpicture}
\qquad
\begin{tikzpicture}[node distance=2.0cm]
\node (A1) {$Morley$};
\node (A2)[right of=A1]{$CR$};
\node (A3)[right of=A2]{$P_0(\mathcal T)$};
\node (B1)[below of=A1]{$S^1_0(\mathcal T)$};
\node (B2)[right of=B1]{$RT_0$};
\node (B3)[right of=B2]{$P_0(\mathcal T)$};
\draw[transform canvas={yshift=0.5ex},->] (A1) --(A2) node[above,midway] {\footnotesize $\operatorname{Curl}_\pw$};
\draw[transform canvas={yshift=0.5ex},->] (A2) --(A3) node[above,midway] {\footnotesize $\div_\pw$};
\draw[transform canvas={xshift=0.5ex},->] (B1)--(B2) node[above,midway] {\footnotesize $\operatorname{Curl}$};
\draw[transform canvas={xshift=0.5ex},->] (B2)--(B3) node[above,midway] {\footnotesize $\operatorname{div}$};
\draw[transform canvas={xshift=0.5ex},->] (A1)--(B1) node[left,midway] {\footnotesize $I_1$};
\draw[transform canvas={xshift=0.5ex},->] (A2)--(B2) node[left,midway] {\footnotesize $I_{RT}$};
\draw[transform canvas={xshift=0.5ex},->] (A3)--(B3) node[left,midway] {\footnotesize $id$};
\end{tikzpicture}
\

\end{center}
\caption{Commuting diagrams. Left: DKT and Bernardi--Raugel, where
         $S^1_0(\mathcal T)$ denotes the standard first-order finite element space.
         Right: Relations for discrete stream functions in the context
         of Morley and Crouzeix--Raviart elements.
         \label{f:complex-br-cr}%
        }
\label{f:commuting}
\end{figure}
In a similar fashion,
several known Stokes pairs result from taking nonconforming
plate elements as discrete stream functions.
The right-hand side differs from the one in the classical 
schemes and leads to pressure-robust modifications.

\begin{example}[Mini element and stabilized Zienkiewicz]
Similarly, the Mini element can be derived from a stabilized
Zienkiewicz element. It is known that the assignment
of $\theta_h$ as the first-order Lagrange interpolation of
$\nabla w_h$ in the DKT context does not lead to a convergent
scheme unless the meshes have a particular structure
\cite{LascauxLesaint1975}.
A possible stabilization is as follows.
The space of discrete gradients is that of first-order
Lagrange elements plus the cubic element bubbles, as is
well known from the Mini element. The assignment
$$
\nabla_h v_h (z) = \nabla v_h (z)
\quad\text{and}\quad
\int_T\nabla_h v_h
=
\int_T\nabla v_h
$$
is unisolvent and it can be checked that it satisfies
properties similar to DKT.
Again, a dimension argument shows that these are 
the discrete stream functions of the Mini element.
\end{example}

\begin{example}[TLLL generates pressure-robust Crouzeix--Raviart]
It is known \cite{FalkMorley1990} that the Morley
element \cite{Morley1968} provides discrete stream functions
for the Crouzeix--Raviart element, similar to the
previous examples.
Various authors have proposed a variation of the Morley element
that foresees to include the first-order Lagrange interpolation
of the method.
This was done by O{\~n}ate, Zarate, and Flores in
\cite{OnateZarateFlores1994} in the context of discrete Kirchhoff
elements (named TLLL therein)
and by Wang, Xu, and Hu \cite{WangXuHu2006} for
the singularly perturbed biharmonic equation.
In the Stokes context, Linke \cite{Linke2014} proposed the
use of the Raviart--Thomas interpolation on the right-hand
side for improving pressure robustness.
It is not difficult to see that these three approaches
are equivalent on simply-connected planar domains.
The TLLL element provides the discrete stream functions for the
Linke discretization.
This follows from the following fact.
The Curl of $I_1 v_h$, the Lagrange interpolation of a Morley
function, equals the RT interpolation of $\nabla_\pw v_h$.
\Cref{f:commuting} illustrates the situation.
\end{example}

\section{Numerical results}\label{s:num}

In this section we briefly report basic numerical results illustrating
the low-regularity error bounds and the performance of the
new three-dimensional DKT-like element.
For $n=3$, the error estimators implemented are based on the simplified evaluation
from \Cref{rem:simplified_est}, see also \Cref{cor:DKT}.

We consider the planar domain
$
\Omega_2 = \left(-1,1\right)^2 
 \setminus \left(\operatorname{conv}\{(0,0),(1,-1),(1,0)\}\right)
$.
Define 
$\omega\coloneqq 7\pi/4$ and $\alpha\coloneqq 0.50500969$. 
The exact singular solution \cite{Grisvard1992}
is given in polar coordinates by
\[
 u_2(r,\theta) 
 = (r^2\cos^2\theta -1)^2\,(r^2\sin^2\theta -1)^2\, r^{1+\alpha}\, g(\theta)
\]
with the function
\begin{equation*}
 g(\theta) =
 \left[ \frac{s_-(\omega) }{\alpha -1} 
              - \frac{s_+(\omega)}{\alpha+1}\right] 
    \big(c_-(\theta)-c_+(\theta)\big)
  - \left[\frac{s_-(\theta) }{\alpha - 1}
                      - \frac{s_+(\theta)}{\alpha+1}\right]
  \big(c_-(\omega)-c_+(\omega)\big),
\end{equation*}
where $s_\pm(t) \coloneqq \sin((\alpha\pm1)t)$
and $c_\pm(t) \coloneqq \cos((\alpha\pm1)t)$.

We solve the discrete biharmonic equation on sequences of uniform
meshes as well as adaptive meshes based on the error estimator
and D\"orfler marking \cite{Doerfler1996} with bulk parameter 1/3.
\Cref{fig:convhist}(left) displays the convergence history in this 
two-dimensional example.
All displayed errors are relative errors.
The symbol $\mathtt{ndof}$ refers to the number of degrees of
freedom.
For illustrating the performance of the three-dimensional generalization
of DKT, we consider the cylinder 
$\Omega_3\coloneqq\Omega_2\times (0,1)$ and the tensor-product
solution $u_3(x,y,z) \coloneqq u_2(x,y) (z-z^2)^2$.
\Cref{fig:convhist}(right) displays the convergence history
on uniform and adaptive meshes.
In both experiments, uniform mesh-refinement leads to convergence
in the $\sigma$ variable
with the asymptotic rate dictated by the singularity exponent of
the exact solution, and higher rates for weaker norms.
We observe that adaptive mesh refinement leads to the optimal
rates $1/n$ with respect to the number of degrees of freedom
for the error in the $\sigma$ variable. The errors in the weaker
norms converge at the doubled rate.
We observe efficiency indices $\eta/\|\sigma-\sigma_h\|$ of around 10.

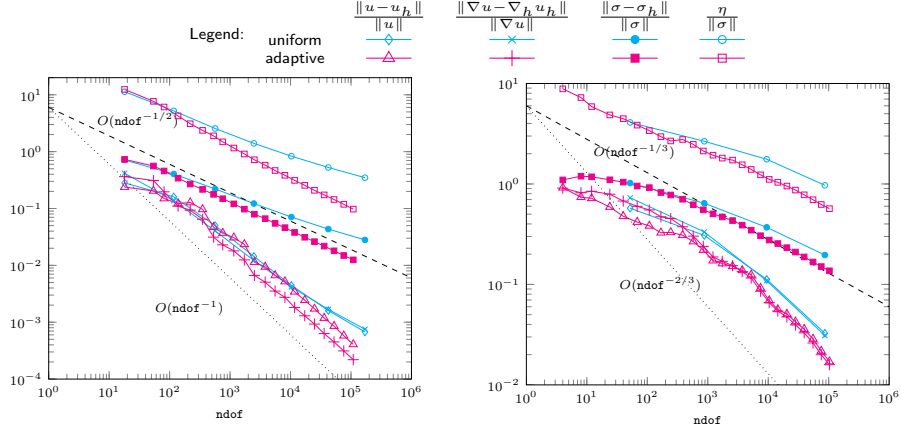
\begin{figure}

{\tiny
\sf Legend:
\begin{tabular}{ccccc}
& $\frac{\|u-u_h\|}{\|u\|}$
& $\frac{\|\nabla u-\nabla_h u_h\|}{\|\nabla u\|}$ 
& $\frac{\|\sigma-\sigma_h\|}{\|\sigma\|}$ 
& $\frac{\eta}{\|\sigma\|}$ 
\\
uniform &
\ref{leg:L2unif} &
\ref{leg:H1unif} &
\ref{leg:H2unif} &
\ref{leg:etaunif}
\\
adaptive &
\ref{leg:L2adapt} &
\ref{leg:H1adapt} &
\ref{leg:H2adapt} &
\ref{leg:etaadapt}
\end{tabular}
}

\begin{minipage}{.49\textwidth}
\pgfplotstableread{
   ndof                     L2                       H1                       H2                       eta
   1.8000000000000000e+01   2.3843784189547257e-01   3.6391395793449438e-01   7.3119590218090891e-01   1.2425933262271487e+01
   5.4000000000000000e+01   2.0406416108180220e-01   3.1046068627563966e-01   5.5922892829522730e-01   7.6783416967081068e+00
   8.1000000000000000e+01   1.4991856145401328e-01   1.9646615456140615e-01   4.5710226985080865e-01   6.0977570406159156e+00
   1.3800000000000000e+02   1.2208555092502507e-01   1.1020758674658099e-01   3.4189733519989796e-01   4.2522923957716507e+00
   2.1900000000000000e+02   1.2756544839732406e-01   9.3489723152171508e-02   2.6906704544364818e-01   3.0983274128508582e+00
   3.5100000000000000e+02   9.6494747829398861e-02   6.4738741054930563e-02   2.1682825416017651e-01   2.3662526242017234e+00
   5.3400000000000000e+02   4.6936958023377502e-02   3.1402278051830858e-02   1.7751022219318102e-01   1.8958868446541930e+00
   7.5600000000000000e+02   3.7028830562696340e-02   2.2946508567724420e-02   1.4858424961275127e-01   1.4872357242304535e+00
   1.1520000000000000e+03   3.0732702346368948e-02   1.7964092862609794e-02   1.2061107727951867e-01   1.1582789814605652e+00
   1.7220000000000000e+03   2.3453006782122712e-02   1.2470096628247390e-02   9.7503686511439883e-02   9.0984513636074327e-01
   2.5230000000000000e+03   1.1420140047118161e-02   6.6547048475692648e-03   7.9417680416658240e-02   7.1986295398600908e-01
   3.8070000000000000e+03   9.3834895998253041e-03   5.0680688332599561e-03   6.6287132077701358e-02   5.7854858551509691e-01
   5.5530000000000000e+03   6.6623069480571300e-03   3.5288130225513900e-03   5.4833169473949769e-02   4.6554440701888322e-01
   8.0700000000000000e+03   5.1949023433169204e-03   2.7418984712004068e-03   4.6025240243349717e-02   3.8360214692038996e-01
   1.1670000000000000e+04   3.4328695276240363e-03   1.8200328978091561e-03   3.7965061674119067e-02   3.1217785230156553e-01
   1.7187000000000000e+04   2.4380916508619176e-03   1.2979089848138301e-03   3.1208356491210020e-02   2.5530893849287994e-01
   2.4861000000000000e+04   1.7227430636837574e-03   9.2825717020372811e-04   2.6104761265728520e-02   2.0882360785905024e-01
   3.5982000000000000e+04   1.1842249648735105e-03   6.3167670353446600e-04   2.1814122132375029e-02   1.7348565832181925e-01
   5.2380000000000000e+04   8.5526175445339771e-04   4.5079629925919332e-04   1.8017230834943526e-02   1.4405425558575780e-01
   7.5942000000000000e+04   5.7929207522136475e-04   3.1307994597176849e-04   1.4914082739642785e-02   1.1847072027528312e-01
   1.0903500000000000e+05   4.0779525960319400e-04   2.1925334311565612e-04   1.2502443580399989e-02   9.7466704048157396e-02

    }\adaptZzwei
\pgfplotstableread{
   ndof                     L2                       H1                       H2                       eta
   1.8000000000000000e+01   2.8216661164331752e-01   4.1647084822436586e-01   7.1393619102274575e-01   1.1335018010456691e+01
   1.1700000000000000e+02   1.6042084565619974e-01   1.3709914262804412e-01   4.0443728943935947e-01   5.2174069512265708e+00
   5.6700000000000000e+02   5.0041266780439235e-02   4.0383981329615813e-02   2.2069401329055321e-01   2.5541538691361287e+00
   2.4750000000000000e+03   1.4385507672058550e-02   1.2185779534389230e-02   1.2219209906356716e-01   1.3987307195958401e+00
   1.0323000000000000e+04   4.4523238731095554e-03   4.1844029874057677e-03   7.0838553564948559e-02   8.3231305985421067e-01
   4.2147000000000000e+04   1.6060601406905030e-03   1.6792524150405083e-03   4.3438301993391874e-02   5.2672383776764775e-01
   1.7030700000000000e+05   6.7319423896814085e-04   7.5223397632554942e-04   2.8022652230530992e-02   3.4834985421885417e-01
    }\unifZzwei
	\begin{tikzpicture}[scale=.7]
		\begin{loglogaxis}[
			ymin=1e-4,ymax=2e1,
			xmin=1e0,xmax=1e6,
			ytick={1e-4,1e-3,1e-2,1e-1,1e0,1e1,1e2},
			xtick={1e0,1e1,1e2,1e3,1e4,1e5,1e6}]
			\pgfplotsset{
				cycle list={%
					{cyan, mark=diamond, mark size=2pt},
					{cyan, mark=x, mark size=2pt},
					{cyan, mark=*, mark size=1.5pt},
					{cyan, mark=o, mark size=1.5pt},
					{magenta, mark=triangle, mark size=2.5pt},
					{magenta, mark=+, mark size=3pt},
					{magenta, mark=square*, mark size=1.5pt},
					{magenta, mark=square, mark size=1.5pt},
				},
				font=\sffamily\scriptsize,
				xlabel=$\mathtt{ndof}$,ylabel=
			}
			\addplot+ table[x=ndof,y=L2]{\unifZzwei};
			\addplot+ table[x=ndof,y=H1]{\unifZzwei};
			\addplot+ table[x=ndof,y=H2]{\unifZzwei};
			\addplot+ table[x=ndof,y=eta]{\unifZzwei};
            \addplot+ table[x=ndof,y=L2]{\adaptZzwei};
			\addplot+ table[x=ndof,y=H1]{\adaptZzwei};
			\addplot+ table[x=ndof,y=H2]{\adaptZzwei};
			\addplot+ table[x=ndof,y=eta]{\adaptZzwei};
  			\addplot+ [color=black,dashed,mark=none] coordinates{(1,6) (1e6,6e-3)};
 			\node(z) at  (axis cs:30,2e0)
 			[above] {$O (\mathtt{ndof}^{-1/2})$};
  			\addplot+ [color=black,dotted,mark=none] coordinates{(1,6e0) (1e6,6e-6)};
 			\node(z) at  (axis cs:1000,1e-3)
 			[above left] {$O (\mathtt{ndof}^{-1})$};
		\end{loglogaxis}
	\end{tikzpicture}
\end{minipage}
\begin{minipage}{.49\textwidth}
\pgfplotstableread{
   ndof                     L2                       H1                       H2                       eta
   0.0000000000000000e+00   9.9372318537394222e-01   9.5521621246723409e-01   1.0119987931786250e+00   1.0829892406137338e+01
   4.0000000000000000e+00   9.1825990100447907e-01   9.0900344266293731e-01   1.0967344916759785e+00   8.8386026905292141e+00
   8.0000000000000000e+00   7.3812378182917715e-01   8.1722205605922416e-01   1.1956877924799341e+00   7.2474525025864027e+00
   1.2000000000000000e+01   7.1990589031665519e-01   8.4582702557696232e-01   1.1773858671025106e+00   5.8847479815673660e+00
   2.4000000000000000e+01   5.8943512269067955e-01   7.9058406661554792e-01   1.1000064101566227e+00   4.8647821526631825e+00
   4.0000000000000000e+01   4.7709853249724243e-01   6.7622728526885567e-01   1.0429025093209616e+00   4.4655688607458544e+00
   6.8000000000000000e+01   4.1756440142986018e-01   5.9934012586324592e-01   9.5372656530668631e-01   3.8518892592658482e+00
   1.0800000000000000e+02   3.8282393815431320e-01   5.5127329653468649e-01   9.1917717597042436e-01   3.3922999746436879e+00
   1.6800000000000000e+02   3.2706466979438276e-01   4.7142105875792312e-01   8.2167505997313883e-01   2.9579977946110252e+00
   2.4400000000000000e+02   3.2959027464935925e-01   4.5292949535778110e-01   7.7714274032647612e-01   2.6797686508109728e+00
   3.8400000000000000e+02   3.0887441666796489e-01   3.7538693213186247e-01   7.0481225589244378e-01   2.7693449698850001e+00
   5.7600000000000000e+02   2.6719484006037580e-01   3.0269803630723696e-01   6.2074544273028798e-01   2.4778777697571814e+00
   8.4400000000000000e+02   2.1986012258043236e-01   2.3915234374858479e-01   5.5245910335120307e-01   2.1290219902213159e+00
   1.2200000000000000e+03   1.7251821115691313e-01   1.8779384007651168e-01   5.0423811220843506e-01   1.9338593431871178e+00
   1.7360000000000000e+03   1.6146289636755570e-01   1.6962025690428406e-01   4.6992000560504393e-01   1.8092066019360502e+00
   2.5720000000000000e+03   1.5156034009223976e-01   1.5388453212428954e-01   4.3190322646946328e-01   1.7281670180312885e+00
   3.7120000000000000e+03   1.3814543433814658e-01   1.3411409212059361e-01   3.8921753386730895e-01   1.5516309997386526e+00
   5.2280000000000000e+03   1.2463598840862440e-01   1.1577921561968177e-01   3.4671463806667918e-01   1.3955321305572754e+00
   7.4680000000000000e+03   9.2162877514004690e-02   8.5520687021800212e-02   3.0550836535217452e-01   1.2324869282217294e+00
   1.0368000000000000e+04   6.9114130363771828e-02   6.5906432579958166e-02   2.7634294864201930e-01   1.1097638768756199e+00
   1.4392000000000000e+04   5.6388963186352004e-02   5.4356797621226367e-02   2.5521468740221770e-01   1.0378006089365301e+00
   2.0364000000000000e+04   5.0735227513570361e-02   4.7488918342679667e-02   2.3228244419235994e-01   9.5342892792510980e-01
   2.8228000000000000e+04   4.2548333116755106e-02   3.9505966716408616e-02   2.1004168365294842e-01   8.6983696635126873e-01
   3.9324000000000000e+04   3.5923441976803000e-02   3.3006126556533753e-02   1.8789222369196351e-01   7.7981539469548089e-01
   5.5064000000000000e+04   2.8653889701086448e-02   2.6116994178066354e-02   1.6724258793222288e-01   6.9659194983769390e-01
   7.5396000000000000e+04   2.1756169642596291e-02   1.9963240505193471e-02   1.5055532461652826e-01   6.2465119640070954e-01
   1.0293200000000000e+05   1.7052055130287562e-02   1.5887374971558701e-02   1.3675922974076948e-01   5.6774589638798711e-01
    }\adaptZdrei
\pgfplotstableread{
   ndof                     L2                       H1                       H2                       eta
   5.2000000000000000e+01   5.7406894230688421e-01   7.2856673306878961e-01   1.0217561562359434e+00   4.0976992067165821e+00
   8.7600000000000000e+02   3.0741787177499791e-01   3.3299332544798554e-01   6.4276690545118476e-01   2.6601385093722563e+00
   9.4360000000000000e+03   1.1215479465844699e-01   1.0887194377426782e-01   3.7026162581924077e-01   1.7609990447670154e+00
   8.6460000000000000e+04   3.2673308986266872e-02   3.1008870471873153e-02   1.9627371017379072e-01   9.6963368041358855e-01
    }\unifZdrei
	\begin{tikzpicture}[scale=.7]
		\begin{loglogaxis}[legend pos=south west,legend cell align=left,
			legend style={fill=none},
			ymin=1e-2,ymax=1e1,
			xmin=1e0,xmax=1e6,
			ytick={1e-4,1e-3,1e-2,1e-1,1e0,1e1},
			xtick={1e0,1e1,1e2,1e3,1e4,1e5,1e6}]
			\pgfplotsset{
				cycle list={%
					{cyan, mark=diamond, mark size=2pt},
					{cyan, mark=x, mark size=2pt},
					{cyan, mark=*, mark size=1.5pt},
					{cyan, mark=o, mark size=1.5pt},
					{magenta, mark=triangle, mark size=2.5pt},
					{magenta, mark=+, mark size=3pt},
					{magenta, mark=square*, mark size=1.5pt},
					{magenta, mark=square, mark size=1.5pt},
				},
				legend style={{fill=none},
					at={(-0.13,.5)}, anchor=east},
				font=\sffamily\scriptsize,
				xlabel=$\mathtt{ndof}$,ylabel=
			}
			\addplot+ table[x=ndof,y=L2]{\unifZdrei};\label{leg:L2unif}
			\addplot+ table[x=ndof,y=H1]{\unifZdrei};\label{leg:H1unif}
			\addplot+ table[x=ndof,y=H2]{\unifZdrei};\label{leg:H2unif}
			\addplot+ table[x=ndof,y=eta]{\unifZdrei};\label{leg:etaunif}
            \addplot+ table[x=ndof,y=L2]{\adaptZdrei};\label{leg:L2adapt}
			\addplot+ table[x=ndof,y=H1]{\adaptZdrei};\label{leg:H1adapt}
			\addplot+ table[x=ndof,y=H2]{\adaptZdrei};\label{leg:H2adapt}
			\addplot+ table[x=ndof,y=eta]{\adaptZdrei};\label{leg:etaadapt}
  			\addplot+ [color=black,dashed,mark=none] coordinates{(1,6) (1e6,6e-2)};
			\node(z) at  (axis cs:10,1.5e0)
 			[above right] {$O (\mathtt{ndof}^{-1/3})$};
  			\addplot+ [color=black,dotted,mark=none] coordinates{(1,6e0) (1e6,6e-4)};
 			\node(z) at  (axis cs:1000,1e-1)
 			[left] {$O (\mathtt{ndof}^{-2/3})$};
		\end{loglogaxis}
	\end{tikzpicture}
\end{minipage}

\caption{Convergence history plots for the two-dimensional example (left)
         and the three-dimensional example (right).}
\label{fig:convhist}
\end{figure}

\section{Conclusive remarks}\label{s:conclusion}
The error analysis of this work can be applied to several
other, classical nonconforming FEM that satisfy (C1)--(C5).
They include but are not limited to the following examples:
The Morley element \cite{Morley1968,WangXu2006};
Fraeijs de Veubeke's elements of type I and II
\cite{FraeijsdeVeubeke1974,LascauxLesaint1975},
the Specht element \cite{Specht1988} and its 
generalization to higher space dimensions
by \cite{WangShiXu2007}, called \emph{New Zienkiewicz Triangle} (NZT) therein;
the Nilssen--Tai--Winther \cite{NilsenTaiWinther2001}.
For these, (C1)--(C2) follow from the degrees of freedom, while (C3)--(C4) are trivial because $\nabla_h = \nabla_\pw$.
Since quadratic polynomials are subset of the local trial spaces, the construction of an interpolation operator satisfying (C5) can follow the proof 
outlined in \Cref{s:dkt}.
The trial functions of Specht and NTW elements are
continuous, leading to convergence rates for the singular perturbation
problem as in \Cref{cor:DKT}.

\bibliographystyle{abbrv}
\bibliography{references}

\end{document}